\documentclass[a4paper,14pt]{amsart}
\usepackage{amssymb,amsfonts,amsthm,amscd,stmaryrd,amsmath}
\usepackage[utf8]{inputenc}
\usepackage[foot]{amsaddr}
\usepackage[dvipsnames]{xcolor} 
\usepackage{graphicx}
\usepackage{tikz}
\usetikzlibrary{decorations.pathreplacing}
\usepackage{subcaption}
\usepackage{color}
\font\script=rsfs10 at 12pt
\usepackage{pgf,tikz}
\usepackage{comment}
 \usepackage{MnSymbol}\usetikzlibrary{arrows}
\usepackage[colorlinks=true,urlcolor=blue,
citecolor=red,linkcolor=blue,linktocpage,pdfpagelabels,
bookmarksnumbered,bookmarksopen]{hyperref}
\usepackage{mathtools}
\mathtoolsset{showonlyrefs}
\usepackage{ae}
\usepackage[utf8]{inputenc}
\numberwithin{figure}{section}

\usepackage{mathrsfs,a4wide}
\usepackage{esint}
\usepackage{color}

\usepackage{import}

\usepackage[leqno]{amsmath}

\def\restrict#1{\raise-.5ex\hbox{\ensuremath|}_{#1}}
 \def\res{\mathop{\hbox{\vrule height 7pt width .5pt depth 0pt \vrule height .5pt width 6pt depth 0pt}}\nolimits}

\newtheorem{thm}{Theorem}[section]
\newtheorem{lem}[thm]{Lemma}
\newtheorem{cor}[thm]{Corollary}
\newtheorem{prop}[thm]{Proposition}

\newtheorem{defn}[thm]{Definition}
\theoremstyle{definition}
\newtheorem{rem}[thm]{Remark}
\newtheorem{counterexample}[thm]{Counterexample}
\theoremstyle{remark}

\newcommand{\abs}[1]{\left\vert#1\right\vert}

\newcommand{\R}{\mathbb{R}}

\newcommand{\Si}{\mathbb{S}}

\newcommand{\Ha}{\mathcal H}
\newcommand{\pa}{\partial}

\def\He{\mbox{\script{H}}\;}
\def\Ie{\mbox{\script{I}}\;}

\DeclareMathOperator{\dive}{div}

\DeclareMathOperator{\dist}{dist}

\DeclareMathOperator{\diam}{diam}
{\left\{\begin{array}{@{}l@{}}}{\end{array}\right.}
\makeatletter %note a di pagina senza numero 1
\def\@makefnmark{} %note a di pagina senza numero 2
\makeatother %note a di pagina senza numero 3

\title[Some isoperimetric inequalities involving the boundary momentum]{Some isoperimetric inequalities involving the boundary momentum}
\author[D.A. La Manna and R. Sannipoli]{Domenico Angelo La Manna and Rossano Sannipoli}
\address{Dipartimento di Matematica e Applicazioni ``R. Caccioppoli'', Universit\`a degli studi di Napoli Federico II \\ Via Cintia, Complesso Universitario Monte S. Angelo, 80126 Napoli, Italy.}
\email{domenicoangelo.lamanna@unina.it}
\address{Dipartimento di Matematica ``Tullio Levi-Civita", Universit\'a degli Studi di Padova, Via Trieste 63, 35131 Padua, Italy.}
\email{rossano.sannipoli@math.unipd.it}
\begin{document}

\begin{abstract}
The aim of this paper is twofold. In the first part we focus on a functional involving a weighted curvature integral and the quermassintegrals. We prove upper and lower bounds for this functional in the class of convex sets, which provide a stronger form of the classical Aleksandrov-Fenchel inequality involving the $(n-1)$ and $(n-2)$-quermassintegrals, and consequently a stronger form of the classical isoperimetric inequality in the planar case. Moreover, quantitative estimates are proved. In the second part we deal with a shape optimization problem for a functional involving the boundary momentum. It is known that in dimension two the ball is a maximizer among simply connected sets when the perimeter and centroid is fixed. We show that the result still holds in the class of undecomposable sets. In higher dimensions the same result does not hold and we consider a new scaling invariant functional that might be a good candidate to generalize the planar case. For this functional we prove that the ball is a stable maximizer in the class of nearly spherical sets in any dimension.
\end{abstract}
\maketitle
\section{Introduction}
In the last few years weighted isoperimetric problems
have attracted the attention of many mathematicians, due to its intrinsic mathematical interest and mainly for its wide class of applications, see \cite{Csa,CGPRS,PS,CMV,Mor,mp} 
for isoperimetric problem with densities and \cite{CGPRS,FuLa} for quantitative weighted isoperimetric inequalities. 
In this paper we are interested in weighted isoperimetric inequalities with density $|\cdot |^2$.
For the weighted boundary integral
\begin{equation}\label{boundmom}
    M(E)=\int_{\partial E}|x|^2\,d\mathcal{H}^{n-1},
\end{equation}
Brock (see \cite{brock2001isoperimetric}) proved 
that $M(E)$ is minimized by the ball centered at the origin when a volume constraint is imposed. Another proof of this fact can be found in \cite{betta2008weighted}, where the authors prove isoperimetric inequalities for a larger class of integrals with radial weights, both on the volume and the perimeter. The quantity introduced in \eqref{boundmom} is known as the \textit{boundary momentum} of $E$. It is worth mentioning that the authors in \cite{bucur2021weinstock} proved an isoperimetric inequality for a functional involving the boundary momentum, the perimeter and the measure, obtaining the Weinstock inequality in any dimension in the class of convex sets.\\
In the first part of the paper, we study a shape optimization problem involving the functional 
\[
\He (E)=\frac1n\inf_{x_0\in \R^n}\int_{\partial E}G_{\partial E}|x-x_0|^2\,d\Ha^{n-1},
\]
where $G_{\partial E} $ is the Gaussian curvature of $\pa E$ and $E$ is a convex set with not empty interior. 
Inequalities involving weighted curvature integrals have been studied in recent years, due to the increasing interest in isoperimetric problems in manifolds. To name some recent contributions, we mention \cite{FuLa1}, where the authors study the case where the weight is the Gaussian density, and
\cite{Kwong_2021, Kwong_2023, zhou2023stability}, where weighted quermassintegrals have been studied.
In fact, in \cite{Kwong_2023} the authors prove that for $E\subset \R^n$ convex and $k \in \{1,\dots ,n-1\}$ it holds
\begin{equation}\label{eq:kwong}
\frac12\int_{\partial  E}|x|^2 H_{\partial E,k}(x)\, d\Ha^1+W_{k-1}(E)\geq c(n,k)W_k(E)^{\frac{n+1-k}{n-k}},
\end{equation}
where $c(n,k)$ is a constant depending only on $n$ and $k$, $W_i(E)$ are the so called quermassintegrals and $H_{\partial E, i}$ are the the $i$-th elementary symmetric polynomials in the $n-1$ principal curvatures of $\pa E$ (for the precise definition see subsection \ref{subsect:curvature}).
\\
Our aim is to bound from below and above a slight modification of the left hand side of \eqref{eq:kwong}. Given $\beta\geq 0$, we define the functional 
\begin{equation} \label{eq:gbeta}
    \mathcal{G}_{\beta}(E)= \He(E)+\beta W_{n-2}(E),
\end{equation} 
and the main result of this first part is the following. 

\begin{thm}\label{prop:weighted1} Let $E\subset \R^n$ be an open, bounded convex set. If $0\le\beta\le n-1$, then we have 
    \begin{equation}\label{eq:gaussmin}
    \beta W_{n-2}(E)+\frac1n \inf_{x_0\in \R^n}\int_{\pa E} |x-x_0|^2   G_{\pa E} \, d\Ha^{n-1} \geq  \frac{(1+\beta)}{\omega_n}W_{n-1}(E)^2. 
    \end{equation}
    If $\beta \geq \beta(n):=(n-1) (1+\tfrac{n}{n+1})$
    \begin{equation}\label{eq:gaussmax}
    \beta W_{n-2}(E)+\frac1n  \inf_{x_0\in \R^n}\int_{\pa E} |x-x_0|^2 G_{\pa E} \, d\Ha^{n-1} \leq  \frac{(1+\beta)}{\omega_n}W_{n-1}(E)^2. 
    \end{equation}
    Moreover, equality holds if and only if $E=B_r(x_0)$ for some $x_0\in \R^n$ and $r>0$.
\end{thm}
This result is telling us 
that for $\beta\leq n-1$ balls are minimizers of $\mathcal{G}_\beta$ among convex sets of fixed $(n-1)$-quermassintegral
(hence the curvature term is the dominating one),
while for $\beta\geq \beta(n)$ the ball is a maximizer in the same class (the $(n-2)$-quermassintegral term is the dominating one).
We also show that, in the planar case, the thresholds $\beta=1$ and $\beta=\beta(2)=5/3$ are sharp (see Corollary \ref{cor:weighted2}). We also provide a quantitative version of Theorem \ref{prop:weighted1}. To be precise, we introduce the functional
\begin{equation}\label{ibeta}
    \Ie_\beta(E)= \frac{1+\beta}{\omega_n}-\frac{\mathcal G_\beta(E)}{W_{n-1}^2(E)}
\end{equation}
and, 
in the spirit of \cite{FMP},
we prove the following quantitative inequality.
\begin{thm} \label{quantitativecurvature}
Let $n\geq 2$. There exist $\delta$ and two positive constants $C_1=C_1(\beta), C_2=C_2(\beta)$ such that, if $E\subset \R^n$ is an open, bounded convex set with $|\Ie_\beta(E)|< \delta$, then
\begin{equation}\label{eq:quanticurv}
    \Ie_\beta(E) \ge C_1 g(\tilde{\mathcal{A}}_{ \mathcal{H}}(E))^\frac52
\end{equation}
when $\beta >\beta(n)$, while for
$\beta <n-1$ we have
\begin{equation} \label{eq:quanticurv1}
 \Ie_\beta(E) \le -C_2 g(\tilde{\mathcal{A}}_{\mathcal{H}}(E))^\frac52.
\end{equation}
where $g:\R_+\to \R_+$ is defined as
\begin{equation*}
g(s)=\begin{cases}
s^2 &\text{if $n=2$}\\
f^{-1}(s^2) &\text{if $n=3$}\\
s^\frac{n+1}{2} &\text{if $n\ge 4$},
\end{cases}
\end{equation*}
with $f(t)=\sqrt{t \log \left(\frac{1}{t}\right)}$ for $0<t<e^{-1}$ and $\tilde {\mathcal A}_{\Ha} (\cdot)$ is the asymmetry index defined in \eqref{eq:haus}.
\end{thm}
Some comments
about the functional $G_\beta(\cdot)$ are in order here. 
\begin{itemize}
\item To prove the lower bound \eqref{eq:gaussmin}, there is no need to be careful with respect to the position of the set in the space. On the other hand, when one is interested in an upper bound as in \eqref{eq:gaussmax}, it is necessary to understand where the set is positioned. This explains why we need to take the infimum in the definition of the functional $\He(\cdot)$. In particular the infimum is attained at a privileged point, that we call curvature centroid, defined in \eqref{eq:centroid1}.
\item
As we point out in Remark \ref{rem:isoperimetric}, the inequality proven in Theorem \ref{prop:weighted1} is stronger than the classical Aleksandrov-Fenchel inequality with indices $n-1$ and $n-2$ (see section \ref{section2}) in $n$ dimensions for convex sets. In particular it gives a stronger form of the classical isoperimetric inequality in the plane for convex sets.
\item
In dimension $2$, the functional $G_\beta(\cdot)$ becomes $\He(\cdot)+\beta\abs{\cdot}$ and we observe the following. By the isoperimetric inequality, balls maximize the volume among sets of fixed perimeter, while the functional $\He(\cdot)$ does not admit a maximizer among convex sets of fixed perimeter, as we highlight in Proposition \ref{rem:nonexistence}. Furthermore, balls are minimizers of $\He(\cdot)$ when the perimeter is fixed.  
\end{itemize}
Looking at the proof of Proposition \ref{rem:nonexistence}, it is clear that the problem comes from the fact that, in order to maximize $\He(\cdot)$, it is convenient to lose mass. For this reason, we introduced the functional $G_\beta(\cdot)$ defined in \eqref{eq:gbeta}, which in two dimension correspond to add a volume penalization to the functional $\He(\cdot)$.
Thus, the functional $G_\beta(\cdot)$ resemble the famous Gamov model for liquid drop, as there is a competition between two functional having opposite behaviors on balls, which is the content of Corollary \ref{cor:weighted2}. 
\\
In the second part of the paper we study the problem of bounding from above the boundary momentum. 
As in two dimension it is possible to bound the diameter of a sets by its perimeter, it is natural to ask if the problem
\begin{equation}\label{maxmomprob}
    \sup_{E \in \mathcal{A}}\{M(E) \;|\; P(E)=m, \, x_E=0\}.
\end{equation}
admits a solution.
In \eqref{maxmomprob}, $\mathcal{A}$ is a certain class of sets, $P(E)$ stands for the perimeter of $E$ and $x_E$ is the centroid of $E$ (see Definition \ref{defn:centroid}).
Clearly, a maximization problem makes no sense in any class of sets fixing only the perimeter, as $M(\cdot)$ positively diverges if we translate $E$ far away from the origin. This leads to study the shape optimization problem nailing the set in a specific point. \\
 \noindent Our first Theorem in this second part is a generalization of an already known result proved in the class of simply connected sets (see \cite[Pages 396-397]{Hu}). We extend it to the class of undecomposable sets of finite perimeter (see Definition \ref{defn:decomposable}) and it is stated as follows.
\begin{thm}\label{thm:undecomposable}
    Let $E\subset \R^2$ be an undecomposable set of finite perimeter.
    Then 
\[
\inf_{x_0\in \R^2}\int_{\pa^* E}|x-x_0|^2\, d\Ha^2
\leq \frac{P(E)^3}{(2\pi)^2}.
\]
\end{thm}

The assumption of $E$ being undecomposable is necessary. Indeed,
if $\{E_n\}$ is a sequence of union of two disjoint balls with the same perimeter and symmetrical with respect to the origin that move away from each other, then the centroid remains fixed at the the origin, but $M(E_n)\to \infty$. So it is necessary that $\mathcal{A}$ is the class of connected and bounded open set in $\mathbb{R}^2$. 
We mention that a related problem has been recently studied in \cite{dietze2023isoperimetric}, where the authors prove \eqref{eq:momentum1} for almost every inner parallel curve of smooth bounded and simply connected sets. This result is relevant to prove an isoperimetric inequality for the first eigenvalue of the magnetic Laplacian with negative Robin boundary condition (see \cite{kachmar2022isoperimetric}).\\

In the last part of this paper we try to generalize Theorem \ref{thm:undecomposable} to higher dimensions. Note that we can construct a sequence of bounded cylinders with fixed perimeter but second momentum as large as we wish (see the counterexample \ref{controesempio}). 
The problem arises from the fact that fixing the perimeter does not guarantee that the diameters of a maximizing sequence are equibounded.
For this reason, we introduce the following scaling and translations invariant functional 
\begin{equation} \label{ndimensionfunctional}
\mathcal{F}(E)= \frac{|E|^{(n-2)(n+1)}}{P(E)^{n^2-1}}\inf_{x_0\in \R^n} \int_{\partial E}|x-x_0|^2\, d\Ha^{n-1}.
\end{equation}

 In Proposition \ref{prop:existence} we prove the existence of a maximizer for \eqref{ndimensionfunctional} in the class of open, bounded and convex sets, and prove that the ball is the unique maximizer in the class of nearly spherical sets (for the definition see Subsection \ref{subsec:nearly}).\\
We also prove a quantitative isoperimetric inequality for the quantity defined in \eqref{ndimensionfunctional}. 

\begin{comment}Before stating the result, we stress that the authors in \cite{brasco2012spectral} prove a quantitative isoperimetric inequality for the minimization problem and, as far as we know, the quantitative inequality for the maximization problem is new.
\end{comment}
\begin{thm}\label{thm:quantitative}
Let $E\subset \mathbb{R}^n$ be a bounded convex set. There exists $\varepsilon>0$ such that if $\mathcal{A}_\Ha(E)\leq \varepsilon$ then
\begin{equation}\label{teo:quantitative}
    \frac{1}{(n\omega_n)^{n^2-2}} -\frac{|E|^{(n+1)(n-2)}\inf_{x_0\in \R^n}\int_{\partial E} |x-x_0|^2\, d\Ha^{n-1}}{P(E)^{n^2-1}} \geq  C(n) g\left(\mathcal{A}_\Ha(E)\right),
\end{equation}
where $\mathcal{A}_{\mathcal H}(E)$ is an asymmetry index defined in \eqref{eq:haus1} and $g$ defined in Theorem \ref{quantitativecurvature}.
\end{thm}
As a corollary of this fact we immediately find a quantitative version of \eqref{eq:momentum1} in two dimensions for convex sets.
\begin{cor}\label{cor:quantitative}
    There exists $\varepsilon>0$ such that for any open bounded convex set $E$, we have
    \[
    \frac{1}{(2\pi)^2}-\frac{\inf_{x_0\in \R^2}\int_{\partial E}|x-x_0|^2\, d\Ha^1}{P(E)^3} \geq C \mathcal{A}_{\mathcal H}(E)^2.
    \]
\end{cor}
The proof of Theorem \ref{thm:quantitative} comes from a standard argument, introduced in \cite{fuglede}, and successfully adapted to weighted isoperimetric inequalities
in presence of a perimeter constraint (see \cite{CiLA,GLPT}) to prove the stability of some spectral inequalities.
\\

The paper is organized as follows. In Section \ref{section2} we fix the notation and report some classical definitions and results that are used in the manuscript. In Section \ref{sec:curvature} we deal with the weighted curvature integral, proving the aforementioned upper and lower bounds, and the quantitative results contained in Theorem \ref{quantitativecurvature}. Section \ref{sec:boundary} is devoted to the study of the isoperimetric inequality involving the boundary momentum and it contains the proofs of Theorems \ref{thm:undecomposable}, \ref{thm:quantitative} and Corollary \ref{cor:quantitative}. In Section \ref{sec:openproblems} we give some comments and list some open problems.

\section{Preliminary results and Definitions}\label{section2}
\subsection{Basic notions and definitions}
Throughout this paper, we denote by $B_R(x_0)$ and $B_R$ the balls in $\mathbb{R}^n$ of radius $R>0$ centered at  $x_0\in\mathbb{R}^n$ and at the origin, respectively. 
Moreover, the $(n-1)$-dimensional Hausdorff measure in $\mathbb{R}^n$ will be denoted by $\mathcal H^{n-1}$ and the Euclidean scalar product in $\mathbb{R}^n$ is denoted by $\langle\cdot,\cdot\rangle$.\\
Let $D\subseteq\mathbb{R}^n$ be an open bounded set and let $E\subseteq\R^{n}$ be a measurable set.  For the sake of completeness, we recall here the definition of the perimeter of $E$ in $D$ (see for instance \cite{AFP,maggi}), that is
\begin{equation*}%\label{def_per_fusco}
P(E;D)=\sup\left\{  \int_E {\rm div} \varphi\:dx :\;\varphi\in C^{\infty}_c(D;\mathbb{R}^n),\;||\varphi||_{\infty}\leq 1 \right\}.
\end{equation*}
The perimeter of $E$ in $\mathbb{R}^n$ will be denoted by $P(E)$ and, if $P(E)<\infty$, we say that $E$ is a set of finite perimeter. 
We denote by $\pa^* E$ the reduced boundary of $E$. 
Moreover, if $E$ has Lipschitz boundary, it holds
\[
P(E)=\mathcal H^{n-1}(\partial E).
\]
The Lebesgue measure of a measurable set $E \subset \mathbb R^{n}$ will be denoted by $|E|$.
We also define the definition of centroid.
\begin{defn}\label{defn:centroid}
Let $E\subset \R^n$ be a set with finite perimeter. We define the \textit{centroid} of $E$ as the barycenter of its boundary, i.e.,
\begin{equation*}
    x_E = \frac{1}{P(E)}\int_{\partial^* E}x\,d\mathcal{H}^{n-1}.
\end{equation*}
\end{defn}
We also recall the definition of decomposable and  undecomposable set.
\begin{defn}\label{defn:decomposable}
A set $E\subset \R^n$ is said a decomposable set of finite perimeter if there exists a partition of $E$ in two measurable sets $E_1$ and $E_2$ with strictly positive measure such that
\[P(E)=P(E_1)+P(E_2).\]
A set is said to be undecomposable if it is not decomposable.
\end{defn}
In two dimensions, decomposable sets of finite perimeter have a strict structure, as explained by the following theorem (see \cite[Corollary 1]{ACL}).
\begin{thm}\label{thm:und}
Let $E$ be a undecomposable set of finite perimeter in $\R^2$. There exists a unique decomposition ($\operatorname{mod}(\Ha^1)$) of $\pa^* E$ into a finite or countable number of Jordan curves $C_0$, $C_i$, $i \in I$ such that $\operatorname{int}(C_i) \subset \operatorname{int} (C_0)$, the sets $\operatorname{int}(C_i)$ are pairwise disjoint and
$P(E)= \Ha^1(C_0) + \sum_i \Ha^1(C_i)$.
\end{thm}
We define the circumradius of $E\subset \mathbb R^{n}$ as 
\begin{equation}\label{exradius_def}
R_E=\inf\{r>0 : E \subset B_r(x), x\in \mathbb R^n\}.
\end{equation} 
We recall the following useful inequality valid for any convex set in dimension two (see \cite{scott2000inequalities})
\begin{equation} \label{exradiusperimeter}
    R_E < \frac{P(E)}{4}.
\end{equation}
The diameter of $E$ is 
\[
\diam (E)=\sup_{x,y \in E} |x-y|.
\]

In order to define the mean curvature, we now introduce some basic tools of differential geometry, such as the tangential grandient and the tangential divergence.

Let $E$ be a set with $C^{2}$ boundary. A vector field $X\in C^1(\partial E,\R^n)$ can be extended in a tubular neighborhood of $\partial E$ and such an extension can be used to define the tangential divergence of $X$ as
\[
\dive_\tau  X= \dive X- \langle DX \nu_{\partial E},\nu_{\partial E}\rangle.
\]
Note that this definition does not depend on the chosen extension.
In the same way a function $u\in C^1(\partial E)$ can be extended in a tubular neighborhood of $\partial E$ and such an extension can be used to define the tangential gradient of $u$ at
a point $x \in \partial E$ 
as
\[
\nabla_\tau u= \nabla u - \langle \nabla u(x), \nu_{\pa E}\rangle \nu_{\pa E}.
\]
For $x\in \partial E$, the normalized mean curvature of $\partial E$ is defined as
\[
H_{\partial E} (x):=\frac{1}{n-1}\dive_\tau \nu_{\partial E}(x).
\]
Observe that with this definition the curvature of the unit ball in $\R^n$ is $1$. The divergence theorem on manifold reads as follows.
\begin{thm} \label{divthm}
	Let $M$ be a $C^2$ surface in $\mathbb{R}^n$ and $V\in C^1_c(\mathbb{R}^n, \mathbb{R}^n)$, we have 
	\begin{equation*}
		\int_{M}\dive_M V\,d\mathcal{H}^{n-1} = (n-1) \int_M V\cdot\bar{H}_M\,d\mathcal{H}^{n-1},
	\end{equation*}
where $\bar{H}_M=H_M\nu_m$ is the vectorial normalized mean curvature, which is the scalar mean curvature $H_M$ in the direction of the outer unit normal $\nu_M$.
\end{thm}

\subsection{Hausdorff distance and curvature measure}\label{subsect:curvature}
We also need to define the Hausdorff distance between two sets.
\begin{defn}
The Hausdorff distance between two non-empty compact sets $E,F\subset\mathbb{R}^n$ is defined by:
\begin{equation*}
\label{disth}
 d^{\mathcal H}(E,F)=\inf \left\{  \varepsilon>0  \; :\; E\subset F+B_{\varepsilon},\; F\subset E+B_{\varepsilon}\right\} .
 \end{equation*}
 If $D\subset\mathbb{R}^n$ is a compact set and $E,F\subset D$ are two bounded open sets, we define the Hausdorff distance between the two open sets $E$ and $F$ by
 $$d_\Ha(E,F):=d^{\mathcal H}(D\setminus E,D\setminus F).$$
 This last definition is independent of the ``big compact box" $D$.
 We say that a sequence of compact (respectively bounded open) sets $(E_j)_j$ converges to the compact (respectively bounded open) set $E$ in the sense of Hausdorff if $d^{\mathcal H}(E_j,E)\to 0$ (respectively $d_{\mathcal H}(E_j,E)\to 0$).
\end{defn}

 Notice that, if $E$ and $F$ are open convex sets, we have
 $$d_\Ha(E,F)=d^\Ha(\overline{E},\overline{F})=d^\Ha (\partial E, \partial F)$$
 and the following rescaling property holds
 \[
 d_\Ha(tE,tF)=t\,d_\Ha (E,F), \quad t>0.
 \]
We define 
\begin{equation}\label{eq:haus1}
\mathcal{A}_{\Ha}=\min_{x\in \R^n} \left\{\frac{d_{\Ha}(E,B_r(x_0)) }{r} ,\,\, P(E)=P(B_r)  \right\}
\end{equation}
and 
\begin{equation}\label{eq:haus}
\tilde{\mathcal A}_{\Ha} = \min_{x\in \R^n} 
\left\{\frac{d_{\Ha}(E,B_r(x_0)) }{r} ,\,\, W_{n-1}(E)=W_{n-1}(B_r)  \right\}    
\end{equation}
so that it is invariant under rigid motion and dilatation.

The Hausdorff convergence is a very important tool to introduce the concept of curvature measure of a generic convex set. Without aiming to completeness, we just introduce the concept and for an exhaustive dissertation we refer to \cite[ Section 1.8]{Sc}. 
For convex sets $E\subset \R^n$ the mean curvature and the Gaussian curvature can be defined also if $E$ is not smooth and they are Radon measure and in fact they are also called curvature measures. We will denote the mean curvature measure and the Gaussian curvature measure respectively by $\mu_E^H$ and $\mu_E^G$. 
Note that when $E$ is a smooth convex set, $\mu^H_E= H_{\partial E} \Ha^{n-1} \res\partial E$, with $ H_{ \partial E}$ defined above.  
We recall the following result, whose proof (actually of a more general result regarding curvature measures) can be found in \cite[Section 4.2]{Sc}.
\begin{thm}\label{thm:shneider}
    Let $\{K_h\}_{h\in \mathbb N}$ be a sequence of convex bodies in $\R^n$ and assume that there exists a convex set $K\subset \R^n$ such that $K_h\to K$ in the Hausdorff sense. Then the curvature measure $\mu_{K_h}\to \mu_{K}$ weakly$^*$ in the sense of measures.
\end{thm}

\subsection{Nearly spherical sets}\label{subsec:nearly} Here we recall the notion of nearly spherical set and state some  useful results.
\begin{defn}\label{defn:nearlysph}
Let $n\geq 2$ and $\varepsilon_0$ a small number. We say that a set $E$ is nearly spherical of mass $m$ if 
$|E|=m$ and 
\begin{equation} \label{eq:nearlydef}
E=\{r x(1+u(x)), \, x\in \mathbb S^{n-1}\}\end{equation}
with $\|u\|_{W^{1,\infty}(\Si^{n-1})}<\varepsilon_0$. 
\end{defn}
Let us notice that $||u||_{L^{\infty}(\Si^{n-1})}^{n-1}=d_\Ha (E,B_r)$, where $B$ is the ball centered at the origin with the same measure as $E$. The Lebesgue measure, perimeter and the boundary momentum can be written in the following way (see \cite{fusco}) 

\begin{equation}\label{MeasPerBoundSupport}
\begin{split}
&\displaystyle \abs{E}= \frac{1}{n}\int_{\Si^{n-1}}(1+u(x))^n\,d\mathcal{H}^{n-1}, \\
&\displaystyle P(E)=\int_{\Si^{n-1}}(1+u(x))^{n-2}\sqrt{(1+u(x))^2+\abs{D_\tau u(x)}^2}\,d\mathcal{H}^{n-1},\\
&\displaystyle \int_E \abs{x}^2\,d\mathcal{H}^{n-1}=\int_{\Si^{n-1}}(1+u(x))^n\sqrt{(1+u(x))^2+\abs{D_\tau u(x)}^2}\,d\mathcal{H}^{n-1}.       
\end{split}
\end{equation}
In next section it will be useful the following lemma (see \cite{fuglede} for a proof).
\begin{lem} \label{interp}
If $v \in W^{1,\infty}(\Si^{n-1})$ and $\displaystyle\int_{\Si^{n-1}}v \,d\Ha^{n-1}=0$, then
\begin{eqnarray}
||v||_{L^{\infty}(\Si^{n-1})}^{n-1} \le 
\begin{cases}
\pi \|\nabla_{\Si^{n-1}} v\|_{L^2(\Si^{n-1})} & n=2\\
4||\nabla_{\Si^{n-1}} v||^2_{L^{2}(\Si^{n-1})} \log \frac{8e ||\nabla_{\Si^{n-1}} v||_{L^{\infty}(\Si^{n-1})}^{n-1}}{||\nabla_{\Si^{n-1}} v||_{L^{2}(\Si^{n-1})}^{2}} & n=3
\\[.2cm]
C(n)  ||\nabla_{\Si^{n-1}} v||_{L^{2}(\Si^{n-1})}^{2} ||\nabla_{\Si^{n-1}} v||_{L^{\infty}(\Si^{n-1})}^{n-3}\,\, & n\ge 4.
\end{cases}
\end{eqnarray}
\end{lem}
\subsection{Some basic notions on convex bodies} 
What follows is contained in \cite[Secions 4 and 8]{Sc}.
\subsubsection{Support functions and curvatures of uniformly convex sets}
Let us consider a convex body $E\in \mathbb R^n$, which is a compact convex set with non-empty interior in $\mathbb R^n$. The support function of a convex body $E\subset \R^n$ is the 1-homogeneous function $h_E: \R^n \to \R$ defined as
\[
h_E(y)= \sup_{p\in E}\{\langle p,y\rangle\}.
\]
The regularity of the support function is strictly connected to the regularity of the boundary of $E$. We say that a convex body $E\in \mathbb R^n$ is of class $C^k$, $k\in \mathbb N$, if its boundary is a $k-$differentiable hypersurface in the sense of differential geometry. If $E$ is of class $C^2$, we can define the \textit{Gauss map} as the $C^1$ function $\nu_E: \partial E \to \Si^{n-1}\subset \mathbb R^n$, that maps every point $x\in \partial E$ to its unique outer unit normal $\nu(x)$. Its differential $W_x:=d_x\nu_E: T_xE \to T_{x}E$, that maps the tangent space $T_x E$ into itself, is a linear map, known as the \textit{Weingarten map}, and its eigenvalues are by definition the \textit{principal curvatures} $\kappa_1,...,\kappa_{n-1}$. We denote by $H_{\partial E,k}$ the $k$-th elementary symmetric polynomial in the $n-1$ principal curvatures $(\kappa_1,...,\kappa_{n-1})$ of $\partial E$, as
\begin{equation}\label{polynomialcurv}
    H_{\pa E, k}= \binom{n-1}{k}^{-1}\sum_{1\le i_1\le i_2\cdots \cdots i_k\le n-1} \kappa_{i_1}\cdots \kappa_{i_k}, \;\; k =1,...,n-1, \qquad H_0=1.
\end{equation}
In particular, $H_{\pa E}:=H_{\partial E,1}$  and $G_{\partial E}:= H_{\partial E, n-1}$ are the normalized mean curvature and the Gaussian curvature of $E$, respectively.\\
The inverse of the Gauss map plays a key role in this paper. In order to define it, we need some extra regularity assumption. 
\begin{defn}
We say that the convex body $E$ is of class $C^2_+$ if it is of class $C^2$ and all the principal curvatures are different from zero at any point (or equivalently that $G_{\pa E}\neq 0$).
\end{defn}
If $E\in C^2_+$, the inverse of the Gauss map $x_{\pa E}:=\nu^{-1}_E:\Si^{n-1}\to \pa E$ is well defined and of class $C^1$. It is known as \textit {reverse Gauss map} and we have 
\[
h_E(y) = \langle y,x_{\pa E}(y) \rangle, \qquad y \in \Si^{n-1}.
\]
This shows that $h_E$ is differentiable on $\Si^{n-1}$, hence on $\R^n\setminus \{0\}$. In particular, $x\in \pa E$ can be characterized as the only one point of $\pa E$ such that
\begin{equation*}
    \nu_E(x_{\pa E}(y)) = \frac{y}{\abs{y}}.
\end{equation*}
Since $y:\R^n\setminus \{0\}\to x_{\pa E}(y)\in \pa E$ is $0-$homogeneous, we have  
\begin{equation*}
    \nabla h_E(y) = x_{\pa E}(y),
\end{equation*}
proving that actually $h_E$ is of class $C^2$. Moreover, by the homogeneity of $h_E$ we have that
\begin{equation} \label{Gaussnablah}
    x_{\pa E}(y) = \nabla_{\Si^{n-1}} h_E(y)+ h_E(y)y.
\end{equation}
The differential of the inverse of the Gauss map $\overline{W_y}:=d_yx_{\pa E}: T_y\Si^{n-1}\to T_y\Si^{n-1}$ is a linear map, known as the \textit{reverse Weingarten map}, and its eigenvalues $r_1,...,r_{n-1}$ are the \textit{principal radii of curvature} of $E$ at $y$. Analogously, we denote by $s_{\partial E,k}$ the $k$-th elementary symmetric polynomial in the $n-1$ principal radii of curvatures $(r_1,...,r_{n-1})$ of $\partial E$, i.e.,
\begin{equation*}
    s_{\pa E, k}= \binom{n-1}{k}^{-1}\sum_{1\le i_1< i_2\cdots \cdots <i_k\le  {n-1}} r_{i_1}\cdots r_{i_k}, \;\; k =1,...,n-1.
\end{equation*}
$H_{\partial E, k}$ and $s_{\partial E,k}$ are connected by the following relationships
\begin{equation}\label{Hksk}
    H_{\partial E,k}(x_{\pa E}(y)) = \frac{s_{\partial E, n-1-k}}{s_{\partial E, n-1}}(y), \qquad k=1,...,n-1
\end{equation}
and
\begin{equation}\label{skHk}
    s_{\partial E,k}(y) = \frac{H_{\partial E, n-1-k}}{H_{\partial E, n-1}}(x_E(y)), \qquad k=1,...,n-1.
\end{equation}
We also recall the following fact (see \cite[Corollary 2.5.3]{Sc}).
\begin{cor}
    Let $E$ be a convex body of class $C^2_+$, so that $h_E$ is of class $C^2$. For $k=1,...,n-1$, we have that $\binom{n-1}{k}s_{\partial E,k}(y)$ is equal to the sum of the principal minors of order $k$ of the Hessian matrix (with respect to an orthonormal basis in $\R^n$) of $h_E$ at $y$.
\end{cor}
In particular, we get  
\begin{equation*}
    s_{\pa E,1}(y) = \frac{\Delta h_E}{n-1},
\end{equation*}
where $\Delta$ denotes the Laplace operator in $\R^n$. While if we choose an orthonormal basis $\{e_1,...,e_n\}$ in $\R^n$, such that $e_n = y \in \Si^{n-1}$ (i.e., $\{e_1,...,e_{n-1}\}$ is an orthonormal basis in $T_y\Si^{n-1}$), then we have
\begin{equation*}
    s_{\pa E,n-1}(y) = \det(\nabla^2 h_E(y)).
\end{equation*}
Equivalently, if we consider the restriction of the support function $h_E$ of $E$ to $\Si^{n-1}$ (always denoted by $h_E$), we have 
\begin{equation*}
    s_{\pa E,1}(y) = \frac{\Delta_{\Si^{n-1}} h_E(y)}{n-1}+ h_E(y),
\end{equation*}
and by \eqref{Gaussnablah} we get
\begin{equation*}
    s_{\pa E,n-1}(y) = \det(\nabla_{\Si^{n-1}} h_E(y)+h_E(y)id),
\end{equation*}
where $\nabla_{\Si^{n-1}} h_E$ and $id$ are the Jacobian of $h_E$ and the identity matrix.\\
In this case, equation \eqref{Hksk} for $k=n-2$ has the nice form
\begin{equation}\label{curvaturen-2}
    H_{\pa E, n-2}(x(y))= \frac{\Delta_{\Si^{n-1}} h_E(y) +(n-1)h_E(y)  }{(n-1)  \operatorname{det}(\nabla_{\Si^{n-1}}h_E(y) +h_E(y) I )}.
\end{equation}
 $H_{\pa E, k}$ and $s_{\pa E, k}$ are useful tools for treating surface integrals. In particular (see \cite[equations 2.61-2.62]{Sc} we have that for any integrable function $f$ on $\pa E$
\begin{equation*}
    \int_{\pa E} f(x)\,d\Ha^{n-1}(x) = \int_{\Si^{n-1}}f(x_E(y))s_{\pa E,n-1}\,d\Ha^{n-1}(y),
\end{equation*}
and for every integrable function $g$ on $\Si^{n-1}$
\begin{equation*}
    \int_{\Si^{n-1}}g(y)\,d\Ha^{n-1}(y)=\int_{\pa E} g(\nu_E(x))H_{\pa E,n-1}\,d\Ha^{n-1}(x).
\end{equation*}

\subsubsection{Quermassintegrals: definition and properties} Let $\emptyset \neq E \subset \mathbb R^n$ be an open, bounded convex set and let $\rho>0$ be a positive real number. We consider the outer parallel set $E_\rho$ as 
\begin{equation*}
    E_\rho:=E + B_\rho = \{x+\rho y \;:\; x\in E, y \in B_1\},
\end{equation*}
where $"+"$ is the Minkowski sum of two sets and $B_1$ is the unit ball centered at the origin.
In particular by Steiner formula we can write the measure of $E_\rho$ in a polynomial in $\rho$ as follows
\begin{equation*}
    \abs{E_\rho}= \sum_{i=0}^n \binom{n}{i}W_{i}(E)\rho^i,
\end{equation*}
where the coefficients $W_i(E)$ are known as \textit{quermassintegrals}. It is well known that these coefficients have an immediate geometrical interpretation when $E$ is $C^2_+$. Indeed, we have
\begin{equation*}
   W_0(E)=|E|,\qquad  n W_i(E) = \int_{\partial E} H_{\pa E, i-1}\,d\Ha^{n-1},\; i=1,...,n,
\end{equation*}
where $H_{\pa E,i-1}$ is defined in \eqref{polynomialcurv}.  
Hence, in this case 
\begin{equation*}
    n W_1(E)= P(E), \;\; n(n-1)W_2(E) = \int_{\partial E} H_{\pa E}\,d\mathcal{H}^{n-1}, \; ...\;, n W_n(E) = \int_{\partial E} G_{\pa E} \,d\Ha^{n-1}, 
\end{equation*}
where $H_{\pa E}$ and $G_{\pa E}$ are respectively the normalized mean curvature and the Gaussian curvature of $\partial E$.
Another useful integral representation is given by  
(see \cite[Section 5.3]{Sc})
\begin{equation}\label{eq:quermass}
W_i(E)=\frac1n \int_{\Si^{n-1}} h_E(y)s_{\pa E,n-1-i}(y)\, d\Ha^{n-1}(y)
\end{equation}
or, using \eqref{skHk}, by
\begin{equation}\label{eq:quermass2}
W_i(E)=\frac1n \int_{\Si^{n-1}} h_E(y)\frac{H_{\partial E,i}(x_E(y))}{H_{\partial E,n-1}(x_E(y))}\, d\Ha^{n-1}(y).
\end{equation}
In the particular case $i=n-1$ \eqref{eq:quermass2} gives, up to a multiplicative constant, the mean width of $\Omega$ and when $i=n-2$ takes the nice form
\begin{equation}\label{eq:quermass1}
    W_{n-2}(\Omega)= \frac1n \int_{\Si^{n-1}}h_E(y)( h_E(y) +\frac{1}{n-1}\Delta_{\Si^{n-1}} h_E(y)    )d\Ha^{n-1}(y).
\end{equation}
Lastly, we recall the well known Aleksandrov-Fenchel inequalities: for any $0\le i<j\le n-1$ we have that
\begin{equation}\label{AF}
 \bigg(\frac{W_j(E)}{\omega_n}\bigg)^\frac{1}{n-j}   \ge \bigg(\frac{W_i(E)}{\omega_n}\bigg)^\frac{1}{n-i}
\end{equation}
and equality holds if and only if $E$ is a ball. Let us stress that for $j=1$ and $i=0$, \eqref{AF} is the classical isoperimetric inequality.

\section{Isoperimetric inequalities involving curvature and boundary momentum}\label{sec:curvature}
In this section we prove an upper and a lower bound for a weighted curvature integral for convex sets. We remark that Proposition \ref{prop:weighted} still holds in higher dimension, provided that we restraint ourselves to the class of convex sets (see Proposition \ref{prop:higherdimensioncurvature}). We start with the lower bound, which is easier as it comes just from an integration by parts after writing the functional $\He (\cdot)$ in polar coordinates. 
\subsection{Lower bounds on the weigthed curvature}
We note that to prove the lower bound no assumptions on the position of the set are needed, as it will be clear in Corollary \ref{cor:weighted}.
\begin{prop}\label{prop:weighted}
    Let $r>0$ and $E\subset \R^2$ be an open $C^{1,1}$ starshaped set with respect to the origin with $|E|=|B_r|$. Then 
    \begin{equation}\label{eq:weighted}
    \int_{\partial E}H_{\partial E}|x|^2\, d\Ha^1 \geq \int_{\partial B_r} H_{\partial B_r} |x|^2\, d\Ha^1 = 2|E|.
    \end{equation}
Moreover, equality holds if and only if $E=B_r$.
\end{prop}
\begin{proof}
     We parametrize the boundary of $E$ by means of the radial function $\rho$, i.e., for $\theta \in [0,2\pi]$ 
     let $\rho:[0,2\pi] \to \R_+$ 
     such that
     $\partial E= \{\rho(\theta)(\cos \theta, \sin \theta),\, \, \theta \in [0,2\pi ]\}.$
     The formulas expressing the area the perimeter and the mean curvature are
	\begin{equation*}
		|E|=\frac{1}{2}\int_0^{2\pi}\rho^2(\theta)\,d\theta,
	\end{equation*}
	\begin{equation*}
		P(E)=\int_0^{2\pi}\sqrt{\rho^2(\theta) + \dot \rho(\theta)^2}\,d\theta
	\end{equation*}
	and
	\begin{equation*}
		H_{\partial E}= \frac{\rho^2+2\dot \rho^2-\rho\ddot \rho}{(\rho^2 + \dot \rho^2)^{\frac{3}{2}}}.
	\end{equation*}
	Hence a simple integration by parts  
	\begin{equation}\label{eq:isocurvature}
	\begin{split}
\int_{\partial E}H_{\partial E}|x|^2\,d\mathcal{H}^{n-1}&= \int_0^{2\pi}\rho^2 \frac{\rho^2+2\dot \rho^2-\rho\dot \rho}{\rho^2 + \dot \rho^2}\,d\theta= \int_0^{2\pi}\rho^2\,d\theta+ \int_0^{2\pi}\rho^2\frac{\dot \rho^2-\rho\ddot \rho}{\rho^2 + \dot \rho^2}\,d\theta \\
		&=2|E|- \int_0^{2\pi}\rho^2 \frac{d}{d\theta}\arctan \bigg(\frac{\dot \rho}{\rho}\bigg)\,d\theta\\
  &=2|E|+2\int_0^{2\pi}\rho\dot \rho\arctan\bigg(\frac
  {\dot \rho}{\rho}\bigg)\ge 2|E|= \int_{\partial B_r}H_{\partial B_r}|x|^2\, d\Ha^1,
 		\end{split}
	\end{equation}
 where we used that $t\arctan t \geq 0$ for all $t\in \mathbb R$.
 If the equality sign holds, then it must be
 \begin{equation*}
     \int_0^{2\pi}\rho\dot \rho\arctan\bigg(\frac
  {\dot \rho}{\rho}\bigg)=0,
 \end{equation*}
 which implies that $\dot\rho=0$ a.e. and hence $\rho=const$.
\end{proof}
As we said before, Proposition \ref{prop:weighted} holds for generic convex sets.
\begin{cor} \label{cor:weighted}
    Let $E\subset \R^2$ be a convex set with not empty interior and denote by $\mu_E$ the curvature measure of $E$. Then 
    \begin{equation}
        \int_{\partial E}|x|^2 d\mu_E \geq \int_{\partial B_r}H_{\partial B_r}|x|^2\, d\Ha^1.
    \end{equation}
\end{cor}
\begin{proof}
    We start observing that \eqref{eq:weighted} holds for $C^{1,1}$ convex sets (hence, we start removing the constraint on the {\it position} of the set). If the origin belongs to the interior of $E$, then Proposition \ref{prop:weighted} holds and, by continuity, it is easy clear that it also holds when the origin belongs to the boundary of $E$. In case $0 \in E^c$ we let $x_0$ be the nearest point to the origin of $E$, hence $|x_0|= \min_{x\in \overline E} |x|$. By convexity we have that $E$ is contained in the half plane tangent to $E$ in $x_0$, which means that  
 $\langle x-x_0 ,x_0 \rangle \geq 0$, or $|x_0|^2 \leq \langle x,x_0\rangle $, for all $x \in \overline E$. Thus 
 we find $$|x-x_0|^2 = |x|^2+|x_0|^2-2\langle x,x_0 \rangle\leq |x|^2-|x_0|^2\leq |x|^2.$$
 Hence 
 $$
\int_{\partial  E}|x|^2 H_{\partial E}(x)\, d\Ha^{1}
\geq \int_{\partial E} |x-x_0|^2 H_{\partial E}(x)\, d\Ha^{1}
= \int_{\partial (x_0+ E)}|x|^2 H_{\partial (x_0+E)} (x)\, d\Ha^1.
 $$
 Observing that obviously $0\in \overline{x_0+E}$,  hence \eqref{eq:weighted} holds, we have the result for $C^1$ convex sets.
 To prove the result for a generic bounded convex set with non empty interior, we recall that for any convex set $E$ there exists a curvature measure $\mu_E$ with $\operatorname{supp}\mu_E \subset \partial E$ such that for any sequence of smooth convex sets $\{E_h\}_{h\in \mathbb N}$ converging in the Hausdorff sense to $E$ we have $H_{\partial E_h}  \res \partial E \rightharpoonup \mu_E$.  
\end{proof} 
\begin{comment}
Finally, we also note that the inequality in Proposition \ref{prop:weighted} can be written in a quantitative form.
\begin{cor}
    Let $E\subset \R^2$ be a $C^{1,1}$ starshaped with respect to the origin. If we parametrize $\partial E$ by means of the radial function $\rho:[0,2\pi]\to \mathbb{R}_+$, i.e., $\partial E = \{ \rho(\theta)(\cos \theta, \sin \theta),\, \theta \in [0,2\pi]\}$, then we have
    \begin{equation}\label{eq:stabilityweighted}
    \int_{\partial E}H_{\partial E}|x|^2\, d\Ha^1 -2|E|\geq 2\int_{\partial E}\bigg( |x|-\frac{\langle x,\nu \rangle^2}{\abs{x}}\bigg)\, d\Ha^1 .
     \end{equation}
\end{cor}
\begin{proof}
We follow the proof of Theorem 1 in \cite{FuLa1}. 
\end{proof}
\end{comment}
We note that in higher dimension the same lower bound continue to be true.

\begin{prop} \label{prop:higherdimensioncurvature}
Let $E\subset\R^n$ be a convex body. 
Then
\[
\int_{\R^n}|x|^2d\mu^H_E\geq n(n-1)|E| 
\]
\end{prop}
\begin{proof}
    The proof of this inequality is just a consequence of the divergence theorem. First, we observe
     that if $E$ is a smooth open convex set and $x\in \partial E$ we have
    \[
    \nabla_{\partial E}|x|= \left(I-\nu_{\partial E}\otimes \nu_{\partial E}\right)\frac{x}{|x|}
    =\frac{x}{|x|}-\left\langle \frac{x}{|x|},\nu_{\partial E}\right\rangle\nu_{\partial E},
    \]
    which in turn implies
    \[
    \langle x,\nabla_{\partial E}|x|\rangle =\frac{1}{|x|}(|x|^2- \langle x,\nu_{\partial E}\rangle^2)\geq 0.
    \]
    The convexity of $E$ implies $H_{\partial E}\geq 0$ and this leads to 
    \[
    \begin{split}
    \int_{\partial E}H_{\partial E}|x|^2\, d\Ha^{n-1}\geq& \int_{\partial E}H_{\partial E}|x|\langle x,\nu\rangle\, d\Ha^{n-1}=
    \int_{\partial E}\operatorname{div}_{\partial E} (x|x|) d\Ha^{n-1}
    \\
    &= (n-1)\int_{\partial E}|x| \, d\Ha^{n-1}
    +\int_{\partial E}\langle x, \nabla_{\partial E}|x|\rangle \, d\Ha^{n-1}
    \\
    &\geq (n-1)\int_{\partial E}|x|\, d\Ha^{n-1}.
    \end{split}
    \]
The result follows by applying the divergence theorem to infer
\[
\int_{\partial E}|x|\, d\Ha^{n-1}
\geq 
\int_{\partial E}\langle x,\nu_{\partial E}\rangle \,d\Ha^{n-1}=n|E|.
\]
The proof for a generic convex set, also in this case, follows verbatim the arguments of Corollary \ref{cor:weighted}.
\end{proof}

\subsection{Upper bound for the weighted curvature}

The second inequality of this section is a bit more delicate. In fact, we are going to prove a sharp upper bound for the functional $\He(\cdot)$.\\
Before proving it, we first define the curvature centroid of a set and prove some elementary properties.
To this end we recall that for a convex set $E$ we denote by $\mu^G_E$ the Gaussian curvature measure associated with $E$.
\begin{defn}\label{defn:curvaturecentroid}
Let $E\subset \mathbb{R}^n$ be an open, bounded set. We define the Gaussian curvature centroid of $E$ the following point
\begin{equation}\label{eq:centroid1}
    x_E^G= \frac{1}{n\omega_n}\int_{\partial E} x \, d\mu^G_E.
\end{equation}
\end{defn}
Of course, when $E$ is smooth enough, say $C^{1,1}$, we have
\[
 x_E^G= \frac{1}{n\omega_n}\int_{\partial E}xG_{\partial E}\, d\Ha^{n-1}.
\]
Using Theorem \ref{thm:shneider} it is immediate to check that that the curvature centroid is continuous with respect to the Hausdorff convergence. 
With this observation in mind, we now prove the next Lemmas.
\begin{lem}
    Let $E\subset \mathbb{R}^n$ be an open bounded convex set. Then 
    \begin{equation*}
        \inf_{y\in\mathbb{R}^n} \int_{\partial E}|x-y|^2d\mu^G_E = \int_{\partial E} |x-x_E^G|^2\, d\mu_E
    \end{equation*}
\end{lem}
\begin{proof}
	Let us define the functional
	\begin{equation*}
		L(y)= \int_{\partial E }\abs{x-y}^2\,d\mu^G_E .
	\end{equation*}
 For $y\not = x_E^G$ we have 
 \[
 L(y)= \int_{\partial E} (|x-x_{E}^G|^2+|y-x_E^G|^2-2\langle y-x_E^G,x-x_E^G\rangle \,d\mu^G_E
 =\int_{\partial E} (|x-x_{E}^G|^2+|y-x_E^G|^2)\,d\mu^G_E,
 \]
 where in the last equality we used \eqref{eq:centroid1}.
The result then easy follows as $L(y) \geq \int_{\partial \Omega} |x-x_{E}^G|^2\,d\mu^G_E$ with equality if and only if $y=x_E^G$.
\end{proof}
When $E$ is convex $X_E^G$ lies inside $E$, as shown in the next Lemma.
\begin{lem} \label{prop:centroid}
Let $E\subset \mathbb{R}^n$ be an open bounded convex set. Then $x_E^G \in E$.
\end{lem}
\begin{proof}
    By contradiction let us suppose that $x_E^G\notin \overline{E}$. Then we can find a positive real number $\delta>0$, such that $\overline{B_\delta(x_E^G)} \cap \overline{E}= \emptyset$. Thus, by Hahn-Banach separation theorem, $\overline{E}$ and $\overline{B_\delta(x_E^G)}$ are strictly separated. This means that there exist a linear map $f:\mathbb{R}^n\to \mathbb{R}$ and a real number $t\in \mathbb{R}$ such that 
    \begin{equation*}
        f(x)<t<f(y)\qquad \forall x\in\overline{E},\, y\in \overline{B_\delta(x_E^H)}.
    \end{equation*}
    Therefore using the Gauss-Bonnet Theorem together with the linearity of $f$ we get
    \begin{equation*}
        f(x_E^H)= f\bigg(\frac{1}{n\omega_n}\int_{\partial E}x \, d\mu^G_E\bigg)=\frac{1}{n\omega_n}\int_{\partial E}f(x)\,d\mu^G_E< t.
    \end{equation*}
    which is clearly a contradiction. 
    To prove that $x_E^G$ is in fact an interior point, we assume that $x_E^G\in \partial E$ and let $l$ be a supporting hyperplane for $E$ at $x_E^G$, i.e., a hyperplane such that $E$ lies on one side of it.
    Since $\He(\cdot)$ is invariant under rotation, without loss of generality we can assume that $x_{E}^G=(x_0,0,\dots,0)$ and $l= \{x=(x_1,x')\in \R \times \R^{n-1}: x_1 =x_0\}$ for some $x_0\in \R$ and $ E \subset \{x_1>x_0\}$. Since $E$ is convex and bounded we have that 
    $\mu_E(\{x_1>x_0\})>0$. Therefore
     $\int_{\partial E} (x_1-x_0)\, d\mu_E>0$ which implies that $(x_0,0,\dots,0)$ does not satisfy \eqref{eq:centroid1}.
\end{proof}
\subsection{Proof of Theorem \ref{prop:weighted1} and consequences.}
\begin{proof} [Proof of Theorem \ref{prop:weighted1}.]
First, we prove the result for a smooth uniformly convex set and then we obtain the general case by approximation.
    Without loss of generality we assume that
    $$
\inf_{x_0\in \R^n}\int_{\pa E} |x-x_0|^2 G_{\pa E} \, d\Ha^{n-1}=\int_{\pa E} |x|^2 G_{\pa E} \, d\Ha^{n-1}.
    $$
Therefore we have that the origin belongs to $E$ by Lemma \ref{prop:centroid}. Let $h: \Si^{n-1}\to \R_+$ be the support function of $E$. Recalling \eqref{Gaussnablah}, we have that
\[
x(\omega)= \nabla_{\Si^{n-1}} h(\omega) +\omega h(\omega),
\]
is the inverse of the Gauss map $x\in \pa E \to \nu(x) \in \Si^{n-1}$. Note that the Gaussian curvature is nothing else that the Jacobian of the Gauss map, therefore the area formula implies
\begin{equation}\label{eq:gaussian}
\int_{\pa E} |x|^2 G_{\pa E}\, d\Ha^{n-1}
= 
\int_{\Si^{n-1}} |\nabla_{\Si^{n-1}} h(\omega)|^2 +h(\omega)^2\, d\Ha^{n-1},
\end{equation}
where we used that $\langle \omega, \nabla_{\Si^{n-1}} h(\omega)\rangle =0$.
Thus by \eqref{eq:quermass1}, \eqref{eq:gaussian} we have
\[
\begin{split}
n(\beta W_{n-2}(E) &+ \int_{\pa E}|x|^2 G_{\pa E}\, d\Ha^{n-1}
-\frac{(1+\beta)}{\omega_n} W_{n-1}(E)^2)
\\
=&
\frac{\beta}{n-1}\int_{\Si^{n-1}}h((n-1)h +\Delta_{\Si^{n-1}} h) \, d\Ha^{n-1}
+ \int_{\Si^{n-1}} h^2+|\nabla_{\Si^{n-1}} h|^2\, d\Ha^{n-1} 
\\
&\qquad\qquad-\frac{1+\beta}{n\omega_n}\left(\int_{\Si^{n-1}}h\right)^2
\\
=&
\int_{\Si^{n-1}} (1+\beta) h^2 +(1-\frac{\beta}{n-1} )|\nabla_{\Si^{n-1}} h|^2  \, d\Ha^{n-1}-\frac{1+\beta}{n\omega_n}\left(\int_{\Si^{n-1}}h\right)^2,
\end{split}
\]
where in the last line we integrated by parts. Let us stress that, for $0\le \beta \le n-1$, \eqref{eq:gaussmin}  follows immediately applying Jensen's inequality in the last line
\begin{equation*}
    \begin{split}
       \beta W_{n-2}(E) &+ \int_{\pa E}|x|^2 G_{\pa E}\, d\Ha^{n-1}
-\frac{(1+\beta)}{n\omega_n} W_{n-1}(E)^2\ge  (1-\frac{\beta}{n-1} )\int_{\Si^{n-1}} |\nabla_{\Si^{n-1}} h|^2  \, d\Ha^{n-1}\ge 0.
    \end{split}
\end{equation*}

 Recall that the space $L^2(\Si^{n-1})$ admits the set of spherical harmonics $\{Y_{k,i}, 1\leq i \leq N_k, k \in \mathbb{N}\}$, i.e., the restriction to $\Si^{n-1}$ of homogeneous harmonic polynomials in $\R^n$, as an orthonormal basis. For $k \in \mathbb{N}$ and $i\leq N_k$, we have
$$
-\Delta_{\Si^{n-1}}Y_{k,i}= k(n+k-2)Y_{k,i}.
$$
Hence, we can write $h$ as
$$
h= \sum_{k=1}^{\infty}\sum_{i=1}^{N_k}a_{k,i}Y_{k,i},
$$
where $a_{k,i}= \int_{\Si^{n-1}}hY_{k,i}\, d\Ha^{n-1}$. Since $\{Y_{k,i}, 1\leq i \leq N_k, k \in \mathbb{N}\}$ is an orthonormal basis we have
$$
\|h\|_{L^2(\Si^{n-1})}^2= a_0^2 + \sum_{k=0}^{\infty}\sum_{i=1}^{N_k}a^2_{k,i},
$$
and using the properties of $Y_{k,i}$ it holds
$$
\|\nabla h
\|_{L^2(\Si^{n-1})}^2= \sum_{k=1}^{\infty}\sum_{i=1}^{N_k}k(n+k-2)a^2_{k,i}.
$$
Note that
\[
a_0^2= \frac{1}{n\omega_n}\left(\int_{\Si^{n-1}}h \, d\Ha^{n-1}\right)^2.
\]
Moreover, since
\[
\inf_{x_0\in \R^n}\int_{\pa E} |x-x_0|^2G_{\pa E}\, d\Ha^{n-1}
=\int_{\pa E}|x|^2G_{\pa E} \, d\Ha^{n-1},
\]
we have that
\begin{equation} \label{eq:quermass3}
0=\int_{\pa E}xG_{\pa E}\, d\Ha^{n-1}
= \int_{\Si^{n-1}}\omega  h(\omega) +\nabla_{\Si^{n-1}}h(\omega)\, d\Ha^{n-1}.
\end{equation}
For $i=1,\dots, n$ we let $e_i$ the standard orthonormal basis of $\R^n$ and   
denote by $(\nabla_{\Si^{n-1}})_j u(\omega)=\langle \nabla_{\Si^{n-1}} u, e_i \rangle$.
It is quickly checked that $(\nabla_{\Si^{n-1}})_ju(\omega) = \operatorname{div}_{\Si^{n-1}} U_i $ where $U_i$ is the vector field 
such that $\langle U_i, e_j\rangle = \delta_{ij} u$.
Therefore, the divergence theorem on $\Si^{n-1}$ yields 
\[
\int_{\Si^{n-1}} \nabla _{\Si^{n-1}} h\, d\Ha^{n-1}
= (n-1 )\int_{\Si^{n-1}} h(\omega) \omega \,d\Ha^{n-1}.
\]
Hence, substituting the above equality in
\eqref{eq:quermass3}, 
we find
\begin{equation}
0= n\int_{\Si^{n-1}} \omega h(\omega)\, d\Ha^{n-1}.  
\end{equation}
Recalling that $Y_{1,i}(\omega)=c_n \omega_i $,
we have 
$a_{1,i}=0$ for $i\in \{1,\dots, n\}$. 
Hence,
\[
\begin{split}
\beta W_{n-2}(E) &+ \int_{\pa E}|x|^2 G_{\pa E}\, d\Ha^{n-1}
-\frac{(1+\beta)}{n\omega_n} W_{n-1}(E)^2
\\
&= 
\sum_{k\geq 2}\sum_{i=1}^{N_k}( (1+\beta) +(1-\frac{\beta}{n-1}) 
k(n+k-2)   )a_{k,i}^2.
\end{split}
\]
If $\beta \geq \beta(n)$, using $k \geq 2$, we have  
\begin{equation}\label{eq:gradient}
\begin{split}
\beta W_{n-2}(E) &+ \int_{\pa E}|x|^2 G_{\pa E}\, d\Ha^{n-1}
-\frac{(1+\beta)}{n\omega_n} W_{n-1}(E)^2
\\
&\leq 
\sum_{k\geq 2}\sum_{i=1}^{N_k}\left(  \frac{1+\beta}{2n} +
1-\frac{\beta}{n-1}\right) 
k(n+k-2)a_{k,i}^2
\\
&= (n+1)(\beta(n)-\beta) \|\nabla h\|^2_{L^2(\Si^{n-1})}.
\end{split}
\end{equation}
     For the general case we can proceed by approximation: given an open convex set $E$, there exists a sequence of smooth strictly convex sets $E_n\subset E$ such that $E_n\to E$ in the Hausdorff sense. This fact can be proved by using the result in \cite{Blocki}, which states that for any bounded convex set $E\subset \R^n$ (actually the result works in every dimension) there exists a smooth strictly convex exhaustion, i.e., a strictly convex function $v: E \to \R$ such that $v \in C^{\infty}(E) \cap C(\overline {E})$ and $v(x)=0$ for $x\in \partial \Omega$.
    Therefore, defining 
    $E_k= \{x\in E : \, v(x)<-\frac1k \}$, we have $E_k$ is a sequence of smooth strictly convex sets. The continuity of all the quantities involved and the weak convergence of the curvature measures with respect to the Hausdorff convergence imply the result for a generic convex set. 
\end{proof}

\begin{rem}\label{rem:isoperimetric}
We remark that
 Theorem \ref{prop:weighted1} implies the classical Aleksandrov-Fenchel inequality \eqref{AF}, when $j=n-1$ and $i=n-2$. Indeed, Theorem \ref{prop:weighted1}, applied with $\beta=0$, gives
\begin{equation}\label{eq:lowerboundG}
    \inf_{x_0}\int_{\pa E} |x-x_0|^2G_{\pa E}\,d\Ha^{n-1} \ge \frac{W^2_{n-1}(E)}{\omega_n}.
\end{equation}
Moreover, for $\beta \geq \beta(n)$, we have
\[
\inf_{x_0}\int_{\pa E} |x-x_0|^2G_{\pa E}\,d\Ha^{n-1} \leq \frac{1+\beta}{\omega_n} W_{n-1}(E)^2-\beta W_{n-2}(E)
\]
which, together with \eqref{eq:lowerboundG}, gives
\[
\frac{W_{n-1}^2(E)}{\omega_n} \geq W_{n-2}(E).
\]
In particular, when $n=2$, the above is exactly the isoperimetric inequality.
\end{rem}

\subsection{Special cases: two and three dimension}
The functional $G_\beta(\cdot)$ in low dimension has an immediate interpetation. We start with $n=2$. 
\begin{cor}\label{cor:weighted2}
Let $\beta>0$ and consider the functional $\mathcal G_\beta$ defined in \eqref{eq:gbeta} and $E\subset \R^2$ a convex set. 
For $\beta\leq 1$ it holds
\begin{equation}\label{eq:corweighted2}
\int_{\R^2} |x|^2\, d\mu_E+ 2\beta |E| \geq \left(1+\beta\right)  \frac{P^2(E)}{2\pi},
\end{equation}
%i.e. the ball is a minimizer among convex sets of fixed perimeter),
while for $\beta \geq \frac53$ we have that
    \begin{equation} \label{eq:corweighted1}
    \inf_{x_0\in \R^2}\int_{\partial E} |x-x_0|^2\, d\mu^G_E+ 2\beta|E| \leq \left(1+\beta \right)\frac{P^2(E)}{2\pi}.
    \end{equation}
    The thresholds $\beta=\frac{5}{3}$ and $\beta=1$ are sharp 
 and 
for $\beta >5/3$ (resp. $\beta<1$) equality in \eqref{eq:corweighted1} (resp. \eqref{eq:corweighted2}) holds if and only if $E=B_{r}(\overline x)$ for some $r>0$ and $\overline x \in \R^2$.
\end{cor}
\begin{proof}
We just have to prove the sharpness of the thresholds.\\
    \textit{Sharpness of the maximality threshold:} Let $\varepsilon>0$ small enough and let us consider the following ellipse $\mathcal E$ of semiaxes $1+\varepsilon$ and $1-\varepsilon$, whose boundary is parametrized for $\theta \in [0,2\pi]$ by
    \begin{equation*}
        \begin{cases}
            x(\theta)= (1+\varepsilon)\cos\theta\\
            y(\theta)=(1-\varepsilon)\sin\theta.
        \end{cases}
    \end{equation*}
    We know that the measure and the curvature of $E$ are given respectively by
    \begin{equation}
    \begin{split}
    &\abs{\mathcal E}= \pi (1-\varepsilon^2),\\
    &H_{\partial \mathcal E}= \frac{1-\varepsilon^2}{[(1+\varepsilon)^2\sin^2\theta+(1-\varepsilon)^2\cos^2\theta]^{\frac{3}{2}}}.
    \end{split}
    \end{equation}
    In particular we have that
    \begin{equation*}
        \begin{split}
            P(\mathcal E)&= \int_0^{2\pi}\sqrt{(1+\varepsilon)^2\sin^2\theta+(1-\varepsilon)^2\cos^2\theta}\,d\theta=\int_0^{2\pi}\sqrt{1-2\cos 2\theta \varepsilon+\varepsilon^2}\,d\theta\\
            &=\int_0^{2\pi}\bigg(1+\frac{1}{2}(-2\cos 2\theta\varepsilon+\varepsilon^2)-\frac{1}{8}(-2\cos 2\theta\varepsilon+\varepsilon^2)^2+o(\varepsilon^2)\bigg)\,d\theta\\
            &=2\pi+\frac{\pi}{2}\varepsilon^2+o(\varepsilon^2).
        \end{split}
    \end{equation*}
     Therefore, we get
    \begin{equation*}
        \frac{P^2(\mathcal E)}{2\pi}= 2\pi+\pi\varepsilon^2+o(\varepsilon^2).
    \end{equation*}
    Moreover, we compute
    \begin{equation*}
    \begin{split}
        \int_{\partial \mathcal E} &H_{\partial E}\abs{x}^2\,d\mathcal{H}^1=(1-\varepsilon^2)\int_0^{2\pi}\frac{(1+\varepsilon)^2\cos^2\theta+(1-\varepsilon)^2\sin^2\theta}{(1+\varepsilon)^2\sin^2\theta+(1-\varepsilon)^2\cos^2\theta}\,d\theta\\
        &=(1-\varepsilon^2)\int_0^{2\pi}\frac{1+2\cos 2\theta \varepsilon+\varepsilon^2}{1-2\cos 2\theta \varepsilon+\varepsilon^2}\,d\theta\\
        &=(1-\varepsilon^2)\int_0^{2\pi}(1+2\cos 2\theta \varepsilon+\varepsilon^2)(1+2\cos 2\theta \varepsilon+(4\cos^22\theta-1)\varepsilon^2+o(\varepsilon^2))\,d\theta\\
        &=(1-\varepsilon^2)\int_0^{2\pi}(1+8\cos^22\theta\varepsilon^2+o(\varepsilon^2))\,d\theta=2\pi+6\pi\varepsilon^2+o(\varepsilon^2).
    \end{split}
    \end{equation*}
    If $B$ is the ball centered at the origin such that $P(B)=P(\mathcal E)$, then
    \begin{equation*}
        \int_{\partial B} H_{\partial B} |x|^2\, d\Ha^1 + 2\beta|B|  -\left(1+\beta\right)\frac{P^2(B)}{2\pi}= 0,
    \end{equation*}
    while for $\varepsilon>0$ small enough and $\beta < 5/3$
    \begin{equation*}
    \begin{split}
        \int_{\partial \mathcal E}& H_{\partial \mathcal E} |x|^2\, d\Ha^1 + 2\beta|\mathcal E|  -\left(1+\beta\right)\frac{P^2(\mathcal E)}{2\pi} \\
        &=2\pi+6\pi\varepsilon^2+2\pi\beta(1-\varepsilon^2)-2\pi(1+\beta)(\varepsilon^2/2+1)+o(\varepsilon^2)=\pi(5-3\beta)\varepsilon^2+o(\varepsilon^2)>0,
    \end{split}
    \end{equation*}
    proving the sharpness of the threshold $\beta = 5/3$.
    \\
    \textit{Sharpness of the minimality threshold}.
    The sharpness of $\beta=1$ comes directly from the proof. In fact for $\beta >1$, let $k_0^2= \lfloor (1+\beta)/(\beta-1)\rfloor +1$,
    where $\lfloor\cdot \rfloor$ denotes the integer part,
    and set $p_1(\theta)= 1 + \frac{\varepsilon}{k^2_0} \sin(k_0\theta)>0$.
    It is immediate to check that for $\varepsilon$ small enough $p_1+\ddot p_1>0$. 
    Thus, let $E_1$ the convex set having $p_1$ as support function.
    We have
    \[
     \int_{\partial E_1} H_{\partial E_1} |x|^2\, d\Ha^1 + 2\beta|E_1|  -\left(1+\beta\right)\frac{P^2(E_1)}{2\pi}
    = ((1+\beta) +(1-\beta)k_0^2)\pi \varepsilon^2<0.
    \]
\end{proof}
The content of the next Proposition is to show that
the functional $\mathcal \He(\cdot)$ in two dimension does not attain a maximizer among open convex sets.
\begin{prop}\label{rem:nonexistence}
Let $l>0$ and $E\subset \R^2$ be a convex set with $P(E)=l$. 
Then
\begin{equation}\label{eq:weighted4}
\He(E)< \frac{\pi}{8}l^2
\end{equation}
and the inequality is sharp.
\end{prop}
\begin{proof}
     Let $E$ be
    any convex set and let $R_E$ be its circumradius. Since $\He$ is invariant under translations we may assume without loss of generality that the smallest ball containing $E$ is centered at the origin. We have
    \begin{equation}\label{eq:sharp}
    \He(E) \leq\int_{\R^2}|x|^2 \, d\mu_E\leq   2\pi R_E^2< \frac{\pi }{8} P^2(E)=\frac{\pi }{8} l^2.
    \end{equation}
    where we used that for any convex set $E$ in $\R^2$ it holds inequality \eqref{exradiusperimeter}: $R_E<\frac{P(E)}{4}$.
    To prove the sharpness of \eqref{eq:weighted4}
    we now construct a sequence of convex sets $E_n$ with $P(E_n)=l$ and 
    $\He(E_n) \to \pi l^2 /8$.
    Let $\alpha \in (0,\pi)$ and consider the points $P_1,P_2,P_3=-P_1$ and $P_4=-P_2$ as the vertices of the rhombus centered at the origin $R_{l,\alpha}$ of perimeter $l$ and with angles $\alpha=\alpha_1$ at the vertices $P_1$ (and $P_3=-P_1$) and $\alpha_2=\pi-\alpha_1$ at vertices $P_2$ (and $P_4=-P_2$).
    One can show that the curvature measure $\mu_{R_{l,\alpha}}$ associated to $R_{l,\alpha}$ is given by 
    \[
    \mu_{R_{l,\alpha}}= \sum_{i=1}^4 (\pi - \alpha_i) \delta_{P_i},
    \]
    where $\delta_{P_i}$ is the Dirac delta centered at $P_i$.
    Therefore by symmetry 
    \[
    \He(R_{l,\alpha})= \sum_{i=1}^4 |P_i|^2 (\pi - \alpha_i)
    = 2 (|P_1|^2(\pi - \alpha) + (\pi - \alpha_2)|P_2|^2
    =2 (|P_1|^2(\pi - \alpha) + \alpha |P_2|^2)
    .
    \]
    Since the rhombus has perimeter equal to $l$ one has that
\[
|P_1|=\frac{l}{4} \cos \frac{\alpha}{2} \quad\text{and} \quad
|P_2|= \frac{l}{4} \sin \frac{\alpha}{2}
\]
which gives
\[
\He(R_{l,\alpha})=\frac{l^2}{8} \left((\pi-\alpha)\cos^2 \frac{\alpha}{2} +\alpha \sin^2 \frac{\alpha}{2}\right):= \frac{l^2}{8} f(\alpha).
\]
It is straightforward to show that $\alpha= \frac{\pi}{2}$ is minimum of $f$ when $\alpha \in [0,\pi]$ and $f(\alpha)\leq f(0)=f(\pi)=\pi $.
Therefore we have constructed a sequence of sets $R_{l,\alpha}$ such that $\He(R_l)\to \frac{\pi}{8}l^2$ as $\alpha \to 0$, which together with \eqref{eq:sharp} proves that the upper bound is sharp and can not be attained.
\end{proof}
We conclude this section
stating explicitely the result in dimension three.
\begin{cor}
Let $\beta>0$ and consider the functional $\mathcal G_\beta$ defined in \eqref{eq:gbeta} and $E\subset \R^3$ a convex set. 
For $\beta\leq 1$ it holds
\begin{equation}\label{eq:corweighted4}
\int_{\R^3} |x|^2\, d\mu^G_E+ \beta  P(E) \geq \left(1+\beta\right)  \frac{(\int_{\R^3} d\mu^H_E)^2}{4\pi^2},
\end{equation}
%i.e. the ball is a minimizer among convex sets of fixed perimeter),
while for $\beta \geq \frac72$ we have 
    \begin{equation} \label{eq:corweighted3}
    \int_{\R^3} |x|^2\, d\mu^G_E+ \beta  P(E) \leq \left(1+\beta\right)  \frac{(\int_{\R^3} d\mu^H_E)^2}{4\pi^2}
    \end{equation}
     Moreover, equality holds if and only if $E=B_{r}(\overline x)$ for some $r>0$ and $\overline x \in \R^3$.
\end{cor}

\subsection{Quantitative version}
In this subsection we prove a quantitative version of the inequality contained in Proposition \ref{prop:weighted1}. To this aim we recall the scaling invariant functional introduced in \eqref{ibeta}. The proof of Proposition \ref{prop:weighted1} leads to the quantitative isoperimetric inequality for $\Ie(\cdot)$.

\begin{proof}[Proof of Theorem \ref{quantitativecurvature}]
We will only prove \eqref{eq:quanticurv} as the proof of the second statement follows identically.
We first use \eqref{eq:gradient}, then \eqref{eq:poicarre} 
 and then 
Lemma \ref{interp} to the function $\tilde h= h-\frac{1}{n\omega_n}\int_{\Si^{n-1}}h\, d\Ha^{n-1}$, to find 
\[\begin{split}
\Ie_\beta(E) &\leq (\beta(n)-\beta)\|\nabla_{\Si^{n-1}} h\|_{L^2(\Si^{n-1})}
\\
&\leq C(n)(\beta(n)-\beta)(n-1)  g(\|\tilde h \|_{L^\infty})
=  C(n)(\beta(n)-\beta)(n-1) g (\tilde{\mathcal{A}}_{\Ha}(E))
\end{split}
\]
\end{proof}
% \begin{rem}
% As in the case of the qualitative inequality, a consequence of the previous proposition is the quantitative isoperimetric inequality in the plane. If we add and subtract in \eqref{finalpropweighted} the following quantity $2\int_0^{2\pi}\dot{p}^2$, then we have
% \begin{equation*}
% \int_0^{2\pi}p^2-\dot p^2 \,d\theta
%     -\frac{1}{2\pi} \left(\int_{0}^{2\pi} p \, d\theta \right)^2 \le -\frac{3}{4}\int_0^{2\pi}\dot p^2\,d\theta,
% \end{equation*}
% that is
% \begin{equation*}
%     P^2(E)-4\pi\abs{E}\ge \frac{3\pi}{2}\int_0^{2\pi}\dot p^2\,d\theta.
% \end{equation*}
% In this way, if we define the scaling invariant functional
% \begin{equation*}
% \Ie_0(E) = 1-4\pi\frac{\abs{E}}{P^2(E)},
% \end{equation*}
% and follow the proof the previous Theorem we have that
% \begin{equation*}
%     \Ie_0(E) \ge \frac{2}{3\pi}d^{\frac{5}{2}}.
% \end{equation*}

\section{Boundary momentum inequalities}\label{sec:boundary}
The aim of this section is to provide upper bounds for the boundary momentum. 
The next proposition essentially gives Theorem \ref{thm:undecomposable} a very particular case, which is a kind of combination of polygons touching in a finite number of points. We refer to Figure \ref{fig1} for a better visualization of the assumptions.
\begin{prop}\label{prop:momentum}
Let $\ell,m\geq 0$ a natural number,
$E_1,\dots E_m$ and $F_1\dots F_\ell$ open polygons such that
\begin{itemize}
\item $F_j \subset \bigcup_{i=1}^m E_i$ for all $j \in \{1,\dots \ell\}$  and $F_i\cap F_j=\emptyset$ for $1\leq i<j\leq \ell$;
\item $E_i \cap E_j= \emptyset$ for all $1\leq i<j\leq m$;
\item $\min_{k\not = i} \{\operatorname{dist}(\pa E_i, \pa E_k)\}=0$  for all $i\in \{1,\dots,m\}$, and $\Ha^1(\pa E_i\cap \pa E_k)=0$ for all $1\leq i<k\leq m$;
%\item $\forall i \in \{1,...,m\}$, $\exists k\in\{1,...,m\}\setminus \{i\}$, such that $\Ha^0(\partial E_i\cap\partial E_k)\in \{1,...,N\}$, with $N\in \mathbb N\setminus\{0\}$;
%\item $\forall j \in \{1,...,l\}$, $\exists i\in\{1,...,m\}$, such that $\Ha^0(\partial E_i\cap\partial F_j)\in \{1,...,N\}$, with $N\in \mathbb N\setminus\{0\}$;
%\item $\Ha^0(\partial 
%E_i\cap\partial F_j)\in \{1,...,N\}$, with $N\in \mathbb N\setminus\{0\}$, $i \in \{1,\dots m\}$ and $j \in \{1,\dots, \ell\}$;
\item $\min_{i} \{\operatorname{dist}(\pa E_i, \pa F_j)\}=0$ for all $j \in \{1,\dots, \ell\}$, and $\Ha^1(\pa E_i\cap \pa F_j)=0$ for all $i \in \{1,\dots m\}$ and $j\in \{1,\dots ,\ell\}$;
\item $\Ha^1(\pa F_i\cap \pa F_j)=0$ for all $i\neq j$, $i,j \in \{1,\dots ,\ell\}$.
\end{itemize}
and let $E=\bigcup_{i=1}^m E_i\setminus \overline{\bigcup_{j=1}^\ell F_j}$. Then
\begin{equation}\label{eq:momentum1}
\inf_{x_0\in \R^2}\int_{\partial E}|x-x_0|^2 \, d\Ha^1 \leq \frac{P(E)^3}{(2\pi)^2}.
\end{equation}
Inequality \eqref{eq:momentum1} holds also if $E$ is a simply connected open set with Lipschitz boundary and equality holds if and only $E=B_r(x_1)$ for some $r>0$ and $x_1 \in \R^2$.
\end{prop} 
\begin{proof} 
We follow the proof in \cite{Hu}.
Without loss of generality, we assume $x_E=0$ and therefore
\[
\inf_{x_0\in \R^2}\int_{\partial E}|x-x_0|^2 \, d\Ha^1 =\int_{\partial E}|x|^2 \, d\Ha^1.
\]
We parametrize $\partial E$ by arc length. Hence, for $s\in [0, P(E)]$ we have
\[
\int_{\partial E}|x|^2 \, d\Ha^1
=\int_{0}^{P(E)}x(s)^2+y(s)^2 ds.
\]
We now use the Fourier decomposition of the periodic functions $x(s),y(s)$ in the interval $[0,P(E)]$
to write
\[
x(s)= a_0 + \sum_{k\geq 1} a_k \cos \left(\frac{2\pi}{P(E)} ks  \right) +b_k \sin \left(\frac{2\pi}{P(E)} k s  \right)
\]
and
\[
y(s)= c_0 + \sum_{k\geq 1} c_k \cos \left(\frac{2\pi}{P(E)} k s  \right) +d_k \sin \left(\frac{2\pi}{P(E)}k s  \right).
\]
with
\[
a_k= \sqrt{\frac{2}{P(E)}}\int_{0}^{P(E)}x(s) \cos \left(\frac{2\pi}{P(E)} ks  \right)\, ds
, \qquad b_k= \sqrt{\frac{2}{P(E)}}\int_0^{P(E)}x(s)\sin \left(\frac{2\pi}{P(E)} ks  \right)\, ds
\]
and the same expression with $y(s)$ in place of $x(s)$ gives the formula of $c_k$ and $d_k$.
Note that the assumption that the centroid of $E$ is at the origin implies $a_0=c_0=0$.
Next we note that
\[\begin{split}
\int_{0}^{P(E)}x(s)^2+y(s)^2 ds
=\sum_{k\geq 1} a_k^2+b_k^2+c_k^2+d_k^2
&\leq 
\sum_{k\geq 1} k^2(a_k^2+b_k^2+c_k^2+d_k^2)
\\
&=\frac{P(E)^2}{(2\pi)^2}\int_{0}^{P(E)}\dot x(s)^2+\dot y(s)^2 ds=
\frac{P(E)^3}{(2\pi)^2}.
\end{split}
\]Moreover, equality holds if and only if all the coefficients $a_k,b_k,c_k,d_k$ are equal to zero for $k>1$, hence if and only if
\[
x(s)=  a_1 \cos \left(\frac{2\pi}{P(E)} s  \right) +b_1 \sin \left(\frac{2\pi}{P(E)}   s  \right)=
 \sqrt{a_1^2+b_1^2} \cos \left(\frac{2\pi}{P(E)}s-\theta_1\right)
\]
and
\[
y(s)= c_1 \cos \left(\frac{2\pi}{P(E)} s  \right) +d_1 \sin 
\left(\frac{2\pi}{P(E)}s  \right)
=\sqrt{c_1^2+d_1^2} \sin \left(\frac{2\pi}{P(E)} s  +\theta_2\right)  
.
\]
with $\theta_1= \arcsin\frac{a_1}{\sqrt{a_1^2+b_1^2}}$
and $\theta_2= \arccos \frac{c_1}{\sqrt{c_1^2+d_1^2}}$
Finally, since $s$ is the arc length we must have
$\sqrt{a_1^2+b_1^2}=\sqrt{c_1^2+d_1^2}$ while
the angles $\frac{2\pi}{P(E)} s  +\theta_2$  
and $ \frac{2\pi}{P(E)} s  -\theta_1$ must be the same for any $s$. This immediately implies that $E$ is a ball.
\end{proof}

\begin{figure} 
 \tikzset{every picture/.style={line width=0.75pt}} %set default line width to 0.75pt        

\begin{tikzpicture}[x=0.75pt,y=0.75pt,yscale=-1,xscale=1]
%uncomment if require: \path (0,295); %set diagram left start at 0, and has height of 295

%Shape: Polygon [id:ds2053259683957922] 
\draw  [fill={rgb, 255:red, 225; green, 228; blue, 229 }  ,fill opacity=1 ] (254,83) -- (349,163) -- (267,218) -- (218,178) -- (204,113) -- cycle ;
%Shape: Polygon [id:ds995546372948827] 
\draw  [fill={rgb, 255:red, 225; green, 228; blue, 229 }  ,fill opacity=1 ] (341,49) -- (404,113) -- (314,133) -- (335,95) -- (271,97) -- cycle ;
%Shape: Right Triangle [id:dp07718046960663294] 
\draw  [fill={rgb, 255:red, 255; green, 255; blue, 255 }  ,fill opacity=1 ] (260.59,140.67) -- (308,190.5) -- (267.1,171.16) -- cycle ;
%Shape: Diamond [id:dp6928669249743873] 
\draw  [fill={rgb, 255:red, 255; green, 255; blue, 255 }  ,fill opacity=1 ] (377.68,87.51) -- (382.1,106.87) -- (362.32,108.49) -- (357.9,89.13) -- cycle ;
%Snip Diagonal Corner Rect [id:dp0441570458135101] 
\draw  [fill={rgb, 255:red, 225; green, 228; blue, 229 }  ,fill opacity=1 ] (180.93,145.5) -- (211,145.5) -- (211,145.5) -- (211,162.07) -- (202.08,171) -- (172,171) -- (172,171) -- (172,154.43) -- cycle ;
%Shape: Polygon [id:ds9113310488451161] 
%Shape: Polygon [id:ds9113310488451161] 
\draw  [fill={rgb, 255:red, 225; green, 228; blue, 229 }  ,fill opacity=1 ] (438,86) -- (488,116) -- (438,196) -- (408,176) -- (392,117) -- (438,166) -- (418,105) -- cycle ;
%Shape: Right Triangle [id:dp20361893733795733] 
\draw  [fill={rgb, 255:red, 255; green, 255; blue, 255 }  ,fill opacity=1 ] (306,73) -- (321,88) -- (306,88) -- cycle ;
%Shape: Polygon [id:ds23126606918156734] 
\draw  [fill={rgb, 255:red, 255; green, 255; blue, 255 }  ,fill opacity=1 ] (463,110) -- (476,110) -- (466,149) -- (448,145) -- (436,129) -- (451,126) -- cycle ;
%Shape: Trapezoid [id:dp4215770430134129] 
\draw  [fill={rgb, 255:red, 255; green, 255; blue, 255 }  ,fill opacity=1 ] (221.5,123) -- (229,98) -- (241,98) -- (248.5,123) -- cycle ;

% Text Node
\draw (217,85) node [anchor=north west][inner sep=0.75pt]  [font=\footnotesize] [align=left] {$E_1$};
% Text Node
\draw (180,133) node [anchor=north west][inner sep=0.75pt]   [align=left] {{\footnotesize $E_2$}};
% Text Node
\draw (297,55) node [anchor=north west][inner sep=0.75pt]   [align=left] {{\footnotesize $E_3$}};
% Text Node
\draw (458,88) node [anchor=north west][inner sep=0.75pt]   [align=left] {{\footnotesize $E_4$}};
% Text Node
\draw (226,124) node [anchor=north west][inner sep=0.75pt]  [font=\footnotesize] [align=left] {$F_1$};
% Text Node
\draw (278,148) node [anchor=north west][inner sep=0.75pt]  [font=\footnotesize] [align=left] {$F_2$};
% Text Node
\draw (315,73) node [anchor=north west][inner sep=0.75pt]  [font=\footnotesize] [align=left] {$F_3$};
% Text Node
\draw (346,101) node [anchor=north west][inner sep=0.75pt]  [font=\footnotesize] [align=left] {$F_4$};
% Text Node
\draw (444,146) node [anchor=north west][inner sep=0.75pt]  [font=\footnotesize] [align=left] {$F_5$};
% Text Node
\draw (59,62) node [anchor=north west][inner sep=0.75pt]   [align=left] {$E=\bigcup_{i=1}^4 E_i\setminus \overline{\bigcup_{j=1}^5 F_J} $};
\end{tikzpicture}

\caption{The set $E$ (in grey) is the set difference between the union of a finite number of open disjoint polygons $E_i$, for which each of them touches at least another one, and the union of a finite number of open disjoint polygons $F_j$, which are contained in $\bigcup E_i$ and must touch its boundary in a finite number of points.}
\label{fig1}
\end{figure}

Before proving Theorem \ref{thm:undecomposable} we give the following definition.
\begin{defn}
Let $E\subset \R^2$ be an undecomposable set of finite perimeter and let $\pa^* E=C_0\cup \bigcup_{i\in I} C_i$, where the parameter set $I$ and the Jordan curves $C_i$ are  the ones provided by Theorem \ref{thm:und}.
For $k\in \mathbb{N}$, we say that $E$ has $k$ interior separate holes if
    $E=\tilde G_0 \setminus \bigcup_{i=1}^k \tilde G_i $, $|\tilde G_i|>0$,
    $\operatorname{dist}(G_i,G_j)>0$ for any $0<i<j\leq k$,  $\operatorname{dist}(G_0^c, G_i)>0$ and the boundary of $G_i$ can be represented by a continuous curve.
\end{defn}

\begin{proof}[Proof of Theorem \ref{thm:undecomposable}.]
Assume that $x_E$ is at the origin.
By Theorem \ref{thm:und} we have that the reduced boundary of $E$  can be written as the union of at most countable many Jordan curves, i.e.,
\[
\pa^*E= C_0 \cup \bigcup_{i\in I} C_i.
\]
For $i=0$ or $i\in I$ we set $G_i=\operatorname{int}(C_i)$ and we denote by $x_i$ the centroid of $G_i$.
\\ \textit{Step one: $G_i$ is a polygon for all $i\geq 0$.}\\
We argue by induction on the number of interior holes.
Let us start with the basis of the induction, i.e., the set $E$ contains one interior separate hole, meaning $E= \tilde G_0 \setminus\tilde G_1$. With an abuse of notation, we also use the name $C_i$ for the curves obtained as the boundary of the interior hole $\tilde G_i$. 
\\
\textit{Case one:} $x_0=0$. In this case we have have $x_1=0$, since $x_E=0$. In fact
\[
0= \int_{\pa E} x\, d\Ha^1= \int_{C_0}x \, d\Ha^1 + \int_{C_1}x    \, d\Ha^1.
\]
Since $C_0,C_1$ are Lipschitz curves, we are in position to apply Proposition \ref{prop:momentum} to the sets $G_0$ and $G_i$ to get
\[
\int_{\pa E}|x|^2\, d\Ha^1=
\int_{C_0}|x|^2\, d\Ha^1+ \int_{C_1}|x|^2\, d\Ha^1
\leq \frac{\Ha^1(C_0)^3 +\Ha^1(C_1)^3}{(2\pi)^2}\leq \frac{P(E)^3}{(2\pi)^2},
\]
where we used the inequality $a^3+b^3\leq (a+b)^3$ and the fact $P(E)= \Ha^1(C_0) +\Ha^1(C_1)$.
\\
\textit{Case two}: $x_0\not =0$.
In this case we observe that for $i=0,1$ it holds
\[
\int_{C_i}|x|^2\,d\Ha^1= \int_{C_i}|x-x_i|^2\, d\Ha^1
+ |x_i|^2 \Ha^1(C_i).
\]
Therefore
\begin{equation}\label{eq:undecomposable}
\int_{\pa E}|x|^2\, d\Ha^1=\sum_{i=0,1}\int_{C_i}|x-x_i|^2\, d\Ha^1 +|x_i|^2 \Ha^1(C_i).
\end{equation}
Now we show that we can slightly rotate the interior hole in such a way that $C_0$ and $C_1$ do not have any parallel side. Indeed,
since $C_1$ is a polygon, the set $\{N \in \mathbb{S}^1: \Ha^1(\{x\in  C_1:\nu_{ C_1}(x) =N   \})>0\} $  is at most countable.
Observe that if we rotate $C_1$ about its centroid the value of the functional does not change. In fact, let $D_1$ the polygon such that  $\pa D_1= x_{1}+ O( C_1 -x_{1})$, where $O$ is a rotation matrix chosen such that 
$C_0$ and $\pa D_1$ do not have parallels sides and $\operatorname{dist}(G_0^c,D_1)>0$.
We have $x_{D_1}=x_{1}$ and
\begin{equation*} 
\begin{split}
\int_{\pa D_1}|x|^2\, d\Ha^1
&= \int_{\pa D_1}|x-x_{D_1}|^2\, d\Ha^1+ |x_{D_1}|^2 \Ha^1(\pa D_1)
\\
&=\int_{C_1}|x-x_{1}|^2\, d\Ha^1+ |x_{1}|^2 \Ha^1(C_1)=\int_{C_1}|x|^2\, d\Ha^1.   \end{split}
\end{equation*}
Hence, setting $E_1= \tilde G_0 \setminus D_1$ we have $x_{E_1}=0$ and
\[
\int_{\pa E}|x|^2\, d\Ha^1= \int_{\pa E_1}|x|^2\, d\Ha^1.
\]
Since $x_{E_1}=x_E=0$ we have $\Ha^1( C_0) x_0=- \Ha^1(\pa D_1) x_1$. Moreover, by using that $\dist(\tilde G_0, \tilde G_1)>0$, for $t>0$ small we consider the curves
\[
 C_{0,t}= C_{0}+t x_0
\,\,\,\ \text{and}\,\,\,
 C_{1,t}=
\pa D_1 +tx_1
\]
and set $\tilde G_{i,t}=\operatorname{int}( C_{i,t})$
and $E_t= \tilde G_{0,t}\setminus \tilde G_{1,t}$. 
Let $ \bar t= \sup\{t: \, \operatorname{dist}( C_{0,t},C_{1,t}) >0  \}$. 
Consider the set
$$
F= \operatorname{int}({C_{0,\bar t}})\setminus\operatorname{int}(C_{1,\bar t}).
$$
Notice that $F$ is such that its boundary is a Lipschitz curve satisfying $x_{F}=0$ 
and $P(F)= P(E)$.
Using $|x_{C_{0,\bar t}}|\geq|x_{0}|$ and $|x_{C_{1,\bar t}}|\geq|x_{1}|$, we infer
\[
\int_{\pa F}|x|^2\, d\Ha^1>\int_{\pa E}|x|^2\,d\Ha^1.
\]
and finally
\[
\int_{\pa E}|x|^2\,d\Ha^1= \int_{\pa E_1}|x|^2\, d\Ha^1\leq \int_{\pa F}|x|^2\, d\Ha^1 \leq \frac{P(F)^3}{(2\pi)^2} = \frac{P(E)^3}{(2\pi)^2}.
\]
This proves the basis of the induction. The induction argument then follows by arguing as above and we briefly explain it: if the centroid of all the curves is at the origin we apply \eqref{eq:momentum1} to each set $\tilde G_i$ and by using the inequality $\sum_i P(\tilde G_i)^3 \leq (\sum_i P(\tilde G_i))^3$ we are done. If there is $\bar i\in \{0, 1,\dots, k\}$ such that the centroid of  $\tilde G_{\bar i}$
is not at the origin, then we eventually rotate the curve $C_i$ so that it does not have any side parallel with the other curves and then use the translation argument provided  before, so that the value of the functional is increased, until they touch. With this procedure, we have decreased by (at least) one the number of interior holes and we use the inductive assumption to get the result. 
\\
We are left to show  the validity of \eqref{eq:momentum1}  when the number of holes is  countable. 
To this aim, we observe that  
$P(E)=l$ implies $\operatorname{diam}(E) \leq \frac l2$, 
as we are in the plane. Since $x_E=0$ we infer that 
\[
|x_i|= \frac{1}{\Ha^1(C_i)} \left|\int_{C_i}x\, d\Ha^1\right|\leq \frac{l}{2}. 
\]
Moreover, given $\varepsilon>0$ there exists $\nu \in \mathbb N$ such that $\sum_{i>\nu}\Ha^1(C_i)<\varepsilon$. Set 
$F_1= \operatorname{int}(C_0)\setminus \overline{\bigcup_{i=1}^{\nu} \operatorname{int}(C_i)}$ and $T= \bigcup_{i> \nu}\operatorname{int}(C_i)$.
We have $E=F_1\setminus \overline{T}$ and
\[
\begin{split}
    \int_{\pa T}|x|^2\, d\Ha^1
    \leq \frac{l^2}{4}  \varepsilon.
\end{split}
\]
Finally, we have
\[\begin{split}
\inf_{x_0 \in \R^2}\int_{\pa E}|x-x_0|^2\, d\Ha^1
&\leq \int_{\pa E}|x-x_{F_1}|^2\, d\Ha^1\\
&= \int_{\pa F_1}|x-x_{F_1}|^2\, d\Ha^1 + \int_{\pa T}|x-x_{F_1}|^2\, d\Ha^1
\\
&\leq \frac{P(F_1)^3}{2\pi} +l^2\varepsilon   \leq  \frac{P(E)^3}{2\pi} +l^2\varepsilon.
\end{split}
\]
As the left hand side does not depend on $\varepsilon$, sending $\varepsilon$ to zero gives \eqref{eq:momentum1}.\\
\textit{Step two: $E$ is a undecomposable set of finite perimeter.}\\
This follows by approximation, as the class of polygonal sets is dense in the class of sets of finite perimeter (see \cite[Theorem II.2.8 and Remark II.2.13]{maggi}) and by  semicontinuity for the boundary momentum (see Lemma 3.5 in \cite{GLPT}).
\end{proof}

Inequality \eqref{eq:momentum1} does not hold in higher dimension, even if we restrict the analysis to convex sets. This is shown by the next counterexample.
\begin{counterexample}\label{controesempio}
Let $L>0$ be a positive number and consider in $\mathbb{R}^3$ the cylinder $C= B^2_{\varepsilon}\times[-L/2,L/2]$, where $B^2_{\varepsilon}$ is the two-dimensional ball centered at the origin of radius $\varepsilon$. Let us chose $P(C)= 2\pi L \varepsilon+2\pi \varepsilon^2=2\pi$, then $L= \varepsilon^{-1}-\varepsilon$. We can write $\partial C = \partial C_l\cup\partial C_+\cup \partial C_-$, where $\partial C_l$ is the lateral surface of $\partial C$ and $\partial C_{\pm}=B_{\varepsilon}(0,0,\pm L/2)$ the basis of $C$. Using cylindrical coordinates 
\begin{equation*}
	\begin{cases}
	x_1=\varepsilon \cos \theta\\
	x_2=\varepsilon \sin \theta \\
	x_3=z,
	\end{cases}
\end{equation*}
with $(\theta,z) \in [0,2\pi]\times [-L/2,L/2]$, then we get as $\varepsilon \to 0$
\begin{equation*}
	\int_{\partial C_l}\abs{x}^2\,d\mathcal{H}^{n-1}= \int_0^{2\pi}\int_{-L/2}^{L/2}\varepsilon^3+\varepsilon z^2\,dz\,d\theta=2\pi \bigg[\varepsilon^3L+\varepsilon L^3/3\bigg]\to +\infty.
\end{equation*}
Meanwhile on the basis, that can be parametrized by
\begin{equation*}
	\begin{cases}
	x_1=r\cos \theta\\
	x_2=r\sin \theta\\
	x_3= \pm L/2,
	\end{cases}
\end{equation*}
with $(r,\theta)\in [0,\varepsilon]\times [0,2\pi]$, we have
\begin{equation*}
\int_{\partial C_{\pm}}|x|^2 d
\mathcal{H}^{n-1}= \int_0^{2\pi}\int_0^{\varepsilon}r^3+rL^2/4\,dr\,d\theta = \frac{\pi}{2} \bigg[\varepsilon^4+ L^2\varepsilon^2\bigg] \to \frac{\pi}{2}.
\end{equation*}
Hence
\begin{equation*}
	\int_{\partial C}|x|^2\,d\mathcal{H}^{n-1}\to \infty,
\end{equation*}
as $\varepsilon \to 0$.
\end{counterexample}

This counterxample hints that some modification of the functional are in order, otherwise there is no hope even in proving that the functional is bounded from above. Therefore we introduce the functional \eqref{ndimensionfunctional} already mentioned in the introduction
\[
\mathcal{F}(E)= \frac{|E|^{(n-2)(n+1)}}{P(E)^{(n-1)(n+1)}}\inf_{x_0\in \R^n} \int_{\partial E}|x-x_0|^2\, d\Ha^{n-1},
\]
which is invariant under translation and dilation.
The reason to study this functional, comes from the following Lemma (see \cite[Lemma 4.1]{esposito2005quantitative}).
\begin{lem}\label{diambound}
Let $E\subset\mathbb{R}^n$ be any open convex set. Then we have that
\begin{equation} \label{eq:diam}
    \diam(E)\le c(n)\frac{P(E)^{n-1}}{|E|^{n-2}},
\end{equation}
where $c=c(n)$ is a positive dimensional constant.
\end{lem}
As an application Lemma \ref{diambound}, we prove that $\mathcal F
:\Omega\subset\mathbb{R}^n\to \mathbb{R}$ has a maximizer.
\begin{prop}\label{prop:existence}
Let $n\geq 3$. There exists a convex set $K$ such that
\[
\sup_{E \,\text{convex}} \mathcal F(E)= \mathcal F (K).
\]
\end{prop}
\begin{proof}
We first show that the functional $\mathcal F$ is bounded among convex sets. To do so, we first note since the functional is scaling and translation invariant we may assume that $|E|=1$ and  
\[
\inf_{x_0\in \R^n} \int_{\partial E}|x-x_0|^2\, d\Ha^{n-1}= \int_{\partial E} |x|^2\, d\Ha^{n-1},
\]
i.e., fixing the centroid of the set $E$ in the origin. Now we use Lemma \ref{diambound}
to infer
\begin{equation}\label{eq:existence}
\mathcal{F}(E) \leq \frac{\operatorname{diam }(E)^2P(E)}{4 P(E)^{(n-1)(n+1)}}
\leq
\frac{C(n)}{P(E)^{n^2-2n}}.
\end{equation}
The isoperimetric inequality $P(E)\geq (n\omega_n)^\frac1n|E|^{\frac{n-1}{n}}$
yields that there exists a constant depending only on the dimension such that 
\[
\mathcal{F}(E) \leq C
\]
for all convex sets $E$.
Now we take a maximizing sequence $\{E_k\}_{k\in \mathbb N}$, i.e., a sequence of convex sets such that
\[
\mathcal{F}(E_k) \geq \sup_{E \,\text{convex}} \mathcal F(E)-\frac1k.
\]
Again, we assume that the sets $E_k$ satisfy $|E_k|=1$ and have centroid at the origin.
We apply the inequality \eqref{eq:existence} to $E_k$ to infer 
\[
\sup_{E \,\text{convex}} \mathcal F(E)-\frac1k \leq 
\frac{C(n)}{P(E_k)^{n^2-2n}},
\]
which immediately implies that $E_k$ is a sequence of convex sets with equibounded perimeters and also by \eqref{eq:diam} we have that $E_k$ are equibounded sets. This implies the existence of a set $E$ such that, up to a subsequence, $|E\triangle E_k|\to 0$ and $P(E_k)\to P(E)$.
Moreover,  this also implies the convergence of the Green measure  relative to $E_k$ to the Green measure relative to $E$ (see \cite[Prop. 1.80]{AFP}). Hence
we get that 
\[
\int_{\partial E_k}|x|^2\, d\Ha^{n-1}\to \int_{\partial E}|x|^2\, d\Ha^{n-1},
\]
i.e., the continuity of the second momentum with respect to the $L^1$-convergence of convex sets. As a consequence we get
\[
\sup_{G \, \text{convex}} \mathcal{F}(G)\leq\lim_k \mathcal{F}(E_k)=\mathcal{F}(E),
\]
which proves the existence of a maximizer among convex sets.
\end{proof}
We now provide a Fuglede-type computation, which guarantees that in any dimension balls centered at the origin are stable local maximizers of $\mathcal F(\cdot)$ among open convex sets. Of course, this is the main tool to prove the stability result in dimension two.\\
In the next proposition we prove that the unit ball centered at the origin is, up to translation, the unique maximizer among nearly spherical set.
\begin{prop}
    Let $n\geq 2$. There exists $\varepsilon_0>0$ such that if $E$ is a nearly spherical set as in Definition \ref{defn:nearlysph} with 
    centroid at the origin and if $\|u\|_{W^{1,\infty}(\Si^{n-1})}<\varepsilon_0$, then
    \[
    \frac{1}{(n\omega_n)^{n^2-2}} -\frac{|E|^{(n+1)(n-2)}\int_{\partial E} |x|^2\, d\Ha^{n-1}}{P(E)^{n^2-1}} \geq  C(n) \|\nabla h\|_{L^2(\Si^{n-1})}^2.
    \]
\end{prop}
\begin{proof}
    Let $E$ a convex set with $|E|=1$ and centroid at the origin and let $u$ the height function with respect to the unit ball, i.e., $\partial E= \{y=x(1+u(x)),\,\, x\in \Si^{n-1}\}$. Looking at \eqref{MeasPerBoundSupport}, the three quantities involved in $\mathcal F (E)$ are given by  
    \begin{equation}\label{eq:momentum}
    \begin{split}
    \int_{\partial E}|x|^2\,d\Ha^{n-1}
    =& \int_{\Si^{n-1}} (1+h)^{n} \sqrt{(1+u)^2+|\nabla u|^2}\, d\Ha^{n-1}
    \\=& 
    n \omega_n 
    +(n+1) \int_{\Si^{n-1}} u\, d\Ha^{n-1}
    +\frac{n(n+1)}{2}\int_{\Si^{n-1}}h^2\, d\Ha^{n-1}
    \\&+\frac12 \int_{\Si^{n-1}} |\nabla u|^2 \, d\Ha^{n-1}
    +o(\|u\|_{W^{1,2}(\Si^{n-1})}^2),
    \end{split}
    \end{equation}
    \begin{equation}
        \label{eq:perimeter}
    \begin{split}
    P(E)=& \int_{\Si^{n-1}} (1+u)^{n-2} \sqrt{(1+u)^2+|\nabla u|^2}\, d\Ha^{n-1}
    \\
     =&n \omega_n 
    +(n-1) \int_{\Si^{n-1}} u\, d\Ha^{n-1}
    +\frac{(n-1)(n-2)}{2}\int_{\Si^{n-1}}u^2\, d\Ha^{n-1}
    \\&+\frac12 \int_{\Si^{n-1}} |\nabla u|^2 \, d\Ha^{n-1}
    +o(\|u\|_{W^{1,2}(\Si^{n-1})}^2),
    \end{split}
    \end{equation}
    and 
    \begin{equation}\label{eq:volume}\begin{split}
    n|E|=&\int_{\Si^{n-1}}(1+u)^n\, d\Ha^{n-1} = n\omega_n +n \int_{\Si^{n-1}}u \ d\Ha^{n-1} + \frac{n(n-1)}{2}\int_{\Si^{n-1}}u^2\, d\Ha^{n-1} 
    \\
    &+o(\|u\|_{W^{1,2}(\Si^{n-1})}^2).
    \end{split}
    \end{equation}
    Since the functional is scaling invariant we assume without loss of generality that $|E|=\omega_n$.
    The assumption $|E|=\omega_n$ togheter with \eqref{eq:volume} gives
    \begin{equation}\label{eq:volume1}
    \int_{\Si^{n-1}} u \, d\Ha^n= -\frac{n-1}{2}\int_{\Si^{n-1}} u^2\, d\Ha^n +o (\|u\|_{L^2(\Si^{n-1})}^2).
    \end{equation}
    Now we use \eqref{eq:perimeter} and \eqref{eq:volume1} to infer
    \begin{equation} \label{eq:perimeter2}
    \begin{split}
       \frac{ P(E)^{n^2-1}}{(n\omega_n)^{n^2-1}}=&
        1 + (n^2-1)\Big(\frac{n-1}{n\omega_n}\int_{\Si^{n-1}}u\, d\Ha^{n-1}+
        \frac{(n-1)(n-2)}{2n\omega_n}\int_{\Si^{n-1}} u^2d\Ha^{n-1}
        \\
        &\qquad\qquad\qquad
        +\frac{1}{2n\omega_n}\int_{\Si^{n-1}}|\nabla u|^2\, d\Ha^{n-1}\Big)
        +o(\|u\|_{W^{1,2}(\Si^{n-1})}^2)
        \\
        =&
         1 + (n^2-1)\left(-\frac{n-1}{2n\omega_n}
        \int_{\Si^{n-1}} u^2d\Ha^{n-1}
        +\frac{1}{2n\omega_n}\int_{\Si^{n-1}}|\nabla u|^2\, d\Ha^{n-1}\right)
        \\
        &+o(\|u\|_{W^{1,2}(\Si^{n-1})}^2).
        \end{split} 
    \end{equation}
Hence, using again \eqref{eq:volume}
the momentum becomes
\begin{equation}\label{eq:momentum2}
\int_{\partial E}|x|^2\, d\Ha^{n-1}=
n \omega_n 
    +\frac{(n+1)}{2}\int_{\Si^{n-1}}u^2\, d\Ha^{n-1}
     +\frac12\int_{\Si^{n-1}}|\nabla u|^2\, d\Ha^{n-1}
     +o(\|u\|_{W^{1,2}(\Si^{n-1})}^2).
\end{equation}
Since we are assuming $|E|=\omega_n$, we just need to estimate 
    \[\begin{split}
    \frac{\mathcal F (E)-\mathcal{F}(B)}{\omega_n^{(n-2)(n+1)}}
    =&\frac{\int_{\partial E}|x|^2\, d\Ha^{n-1}}{P(E)^{n^2-1}} -
    \frac{1}{(n\omega_n)^{n^2-2}}
    \\
    =&
    \frac{ 
    (n\omega_n)^{n^2-2}\int_{\partial E}|x|^2\, d\Ha^{n-1}
    -P(E)^{n^2-1}}{(n\omega_n)^{n^2-2} P(E)^{n^2-1}}.
    \end{split}
    \]
 We now use \eqref{eq:perimeter2} and \eqref{eq:momentum2} to get
\begin{equation}\label{eq:fuglede}
\begin{split}
    \int_{\partial E}|x|^2\, d\Ha^{n-1}-\frac{ 
    P(E)^{n^2-1}}{(n\omega_n)^{n^2-2}}
    =&
    \left(\frac{n+1}{2} +\frac{(n-1)^2(n+1)}{2}\right)\int_{\Si^{n-1}}u^2\,d\Ha^{n-1}
    \\
    &-\frac{ (n^2-1) -1}{2}\int_{\Si^{n-1}}|\nabla u|^2\, d\Ha^{n-1} +o(\|u\|^2_{W^{1,2}(\Si^{n-1})}).
\end{split}
\end{equation}
 Now we exploit the assumption that the centroid of $E$ coincides with the origin. We now decompose $h$ in terms of the spherical harmonics already defined in the previous section.
Note that the harmonic polynomials corresponding to $k=0$ are constants while $Y_{1,i}=x_i$ for $i\leq n$.
Thus \eqref{eq:volume1} gives
\[
|a_0|^2=o(\|u\|_{L^2(\Si^{n-1})}^2),
\]
while exploiting 
\[
0=\int_{\partial E} x_i \, d\Ha^{n-1}
= \int_{\Si^{n-1}} x_i (1+h)^{n-1} \sqrt{(1+u)^2+|\nabla u|^2
}\, d\Ha^{n-1},
\]
we get a smallness condition on the coefficients corresponding to $k=1$, which reads as
\[
|a_{1,i}|^2 = o(\|u\|^2_{W^{1,2}(\Si^{n-1})}).
\]
Therefore, we arrive at 
\begin{equation}\label{eq:poicarre}
\begin{split}
2n\| u\|_{L^2(\Si^{n-1})}^2 =&o(\|u\|^2_{W^{1,2}(\Si^{n-1})})
+2n\sum_{k=2}^{\infty}\sum_{i=1}^{N_k}a^2_{k,i}
\\
\leq&
\sum_{k=2}^{\infty}\sum_{i=1}^{N_k}k(n+k-2)a^2_{k,i}
+o(\|u\|^2_{W^{1,2}(\Si^{n-1})})
\\
=& \|\nabla u\|_{L^2(\Si^{n-1})}^2 +   o(\|u\|^2_{W^{1,2}(\Si^{n-1})}).
\end{split}
\end{equation}
A combination of \eqref{eq:fuglede} and \eqref{eq:poicarre} 
gives
\[
\begin{split}
 \int_{\partial E}|x|^2\, d\Ha^{n-1}
    -\frac{P(E)^{n^2-1}}{(n\omega_n)^{n^2-2}}
    \le&\frac{-n^3-n^2+4n+2}{4n} 
    \int_{\Si^{n-1}}|\nabla u|^2\, d\Ha^{n-1}
    \\
    &+o(\|u\|_{W^{1,2}(\Si^{n-1})})
    \\
    =& -C(n)\int_{\Si^{n-1}}|\nabla u|^2\, d\Ha^{n-1} +o(\|u\|^2_{W^{1,2}(\Si^{n-1})}),
    \end{split}
\]
which yields the result.
\end{proof}
Note that in dimension two we have 
$C_2=\frac{1}{4}$.
We now prove the quantitative result stated in Theorem \ref{thm:quantitative}.

\begin{proof} [Proof of Theorem \ref{thm:quantitative}] Since the functional is invariant under translations and dilatations, we assume that the centroid of $E$ is at the origin and $P(E)=P(B)$.
Hence we have that 
$$
\partial E=\{ x(1+u(x)),\,\,x\in \Si^{n-1}\} 
$$
Let $\rho$ such that $|B_{\rho}|=|E|$.
Since $E$ and $B_{\rho}$ have the same measure, we have
\begin{equation} \label{area}
\frac{1}{n} \int_{\Si^{n-1}}(1+u)^nd\Ha^{n-1}=\omega_n \rho^n.
\end{equation}
Formula \eqref{area} leads to
$$
1-\rho^n=- \frac{1}{\omega_n}\int_{\Si^{n-1}} \sum_{k=1}^n \binom{n}{k} u^k d\Ha^{n-1},
$$
which implies $|1- \rho|<C(n)\|u\|_{\infty}$. Let $v=\rho(1+u)^n-1$. Then, it holds $$C_3\|v\|_{\infty}\leq \|u\|_\infty \leq C_4 \|v\|_\infty,$$ where $C_3$ and $C_4$ are constant depending only on the dimension.
Moreover, the Leibniz rule yields $D_\tau v = n\rho^n (1+u)^{n-1} D_\tau u$ and thus
$$
C_5 \|D_\tau u\|_{2}\leq \|D_\tau v\|_{2} 
\leq C_6 \|D_\tau u \|_{2},
$$
where $C_5,C_6$ depend on the dimension and on $\rho$.
From \eqref{area}, we know that  $v$ has zero integral. Thus, we can apply Lemma \ref{interp} to $v$ and use the above inequality to infer
\begin{eqnarray*}
||v||_{\infty} \le 
\begin{cases}
\pi \|D_\tau v\|_{2} & n=2\\
4\|D_\tau v\|^2_2\log \frac{8e \|D_\tau v\|^{n-1}_{\infty}}{||D_\tau v||_{2}^{2}} & n=3
\\[.2cm]
C(n)  ||D_\tau v||_{2}^{2} ||D_\tau v||_{\infty}^{n-3}\,\, & n\ge 4.
\end{cases}
\end{eqnarray*}
Recalling that $\|u\|_{\infty}= \mathcal {A}_\mathcal{\Ha} (E)$, we get the result.
\end{proof}
\begin{proof}
    [Proof of Corollary \ref{cor:quantitative}]
    Since we are dealing with an isoperimetric problem involving the boundary momentum with a perimeter constraint, the result can be proven by following the argument of the main result in \cite{GLPT} (in higher dimension) and we just highlight the main steps.
\\
{\it Step one: Continuity of the functional}.
The first thing we need to know is that given a sequence of convex sets $\{E_n\}_{n\in \mathbb N}\subset \R^2$, with $P(E_n)=L$, there exists  a convex set $E$ with $P(E)=L$ and 
\[
\inf_{x_0\in \R^2}\int_{\partial E}|x-x_0|^2\, d\Ha^1 =\lim_n \inf_{x_0\in \R^2}\int_{E_n}|x-x_0|^2\, d\Ha^1.  
\]
Note that we proved it already in Proposition \ref{prop:existence}.
\\
{\it Step two: Qualitative result}. As a byproduct of the continuity shown in step one 
and Proposition \eqref{prop:momentum} (the second part of the statement),
it is easy to check that for any $\varepsilon>0$ there exists $\delta>0$ such that if $P(E)=L$  and 
\begin{equation}\label{eq:quanti2}
\frac{1}{(2\pi)^2}  -   \frac{\inf_{x_0\in \R^2}\int_{\partial E}|x-x_0|^2}{P(E)^3}\leq \delta,
\end{equation}
then $A_{\mathcal H}(E)\leq \varepsilon$.
\\
{\it Step three: Conclusion.} From step one and step two, we know that we can chose $\delta>0$ small enough such that any convex set  satisfying \eqref{eq:quanti2} also
satisfies $\mathcal {A}_\Ha(E)\leq \varepsilon$. 
The conclusion follows directly from Theorem \ref{thm:quantitative}.
\end{proof}

\section{Final comments and open problems} \label{sec:openproblems}
In this section we want to give some comments and list the questions that are still open and that can be explored.
\begin{enumerate}
    \item In Theorem \ref{thm:quantitative} we prove that in any dimension the ball is a local maximizer of the functional $\mathcal F(\cdot)$ defined in \eqref{ndimensionfunctional}. Is it true that balls are global maximizers of $\mathcal F(\cdot)$ among convex sets?
    \item Are the thresholds $\beta=n-1$ and $\beta=\beta(n)$ in Theorem \ref{prop:weighted} sharp?
    \item Let $n=2$, $l>0$ and $\beta\ge 0$. Let us consider the functional $\mathcal G_\beta $ defined in \eqref{eq:gbeta}.
Denote by $\mathcal{C}_l$ the class of convex sets of perimeter $l$, i.e.,
\[
\mathcal{C}_l= \{F\subset \R^2: F \text{ is convex and } \, P(F)=l \}.
\]
In corollary \ref{cor:weighted2} we proved that the ball is an extremal set in $\mathcal{C}_l$ when $0<\beta\le 1$ and $\beta \ge \frac53$, while in Proposition \ref{rem:nonexistence} we observed that there are no maximizers when $\beta=0$. \\
 Does it exist $\beta_1 \in (0,5/3)$ such that the functional $G_\beta$ does not attain maximizers in $\mathcal C_l$ for $\beta \leq \beta_1$? Does it exist $\beta_2 \in (1,\infty)$, such that $\mathcal{G}_\beta$ does not attain minimizers in $\mathcal C_l$ for $\beta \ge \beta_2$?
\end{enumerate}

\section*{Acknowledgments}
Both the authors are member of GNAMPA and the project has been partially supported by GNAMPA. The second author acknowledges the MIUR Excellence Department Project awarded to the Department of Mathematics, University of Pisa, CUP I57G22000700001. 
\section*{Conflicts of interest and data availability statement}
The authors declare that there is no conflict of interest. Data sharing not applicable to this article as no datasets were generated or analyzed during the current study.

\bibliographystyle{plain}
\bibliography{biblio}

\end{document}